\documentclass[11pt]{article}
\usepackage[utf8]{inputenc}

\usepackage[top=1in, bottom=1in, left=1in, right=1in]{geometry}

\usepackage[utf8]{inputenc} %
\usepackage[T1]{fontenc}    %

\usepackage{url}            %
\usepackage{booktabs}       %
\usepackage{amsfonts}       %
\usepackage{nicefrac}       %
\usepackage{microtype}      %
\usepackage{xcolor}         %
\usepackage[skip = 3.5pt plus1pt, indent = 17pt]{parskip}        %

\usepackage[american]{babel}
\usepackage[utf8]{inputenc} %
\usepackage[T1]{fontenc}    %
\usepackage{csquotes}
\usepackage{hyperref}       %
\usepackage{url}            %
\usepackage{booktabs}       %
\usepackage{amsfonts}       %

\usepackage{nicefrac}       %
\usepackage{microtype}      %
\usepackage{xcolor}         %
\usepackage{array}          %
\usepackage{tikz}
\usetikzlibrary{calc,positioning,decorations.pathreplacing,shapes}

\usepackage{amsmath}
\usepackage{amsthm}
\usepackage{amssymb}
\PassOptionsToPackage{normalem}{ulem}
\usepackage{ulem}
\usepackage{dsfont}
\usepackage[noabbrev, capitalise]{cleveref}
\usepackage{thmtools} 
\usepackage{thm-restate}

\usepackage[hide=false,setmargin=true,marginparwidth=0.8in]{marginalia}

\usepackage{enumitem}
\usepackage{xspace}
\usepackage{graphicx}
\usepackage{subcaption}

\usepackage[%
    minnames=1,maxnames=99,maxcitenames=2,
    style=numeric-comp,
    sorting=none,
    sortcites=true,
    doi=true,
    url=true,
    uniquename=init,
    giveninits=true,
    backend=biber,
    hyperref=auto,
    natbib=true]{biblatex}

\DeclareFieldFormat{bibentrysetcount}{\mkbibparens{\mknumalph{#1}}}
\DeclareFieldFormat{labelnumberwidth}{\mkbibbrackets{#1}}

\renewbibmacro{in:}{%
  \ifentrytype{article}{}{\printtext{\bibstring{in}\intitlepunct}}}
\renewbibmacro{in:}{%
  \ifentrytype{inproceedings}{}{\printtext{\bibstring{in}\intitlepunct}}}

\DeclareNameAlias{sortname}{family-given}

\theoremstyle{plain}
\newtheorem{theorem}{\protect\theoremname}[section]
  \theoremstyle{plain}
  \newtheorem{proposition}[theorem]{\protect\propositionname}
  \theoremstyle{plain}
  \newtheorem{corollary}[theorem]{\protect\corollaryname}
  \theoremstyle{remark}
  \newtheorem{remark}[theorem]{\protect\remarkname}
  \theoremstyle{plain}
  
  \theoremstyle{plain}
  \newtheorem{lemma}[theorem]{\protect\lemmaname}
  \theoremstyle{plain}

\usepackage{mathtools}

\usepackage{contour}
\usepackage[normalem]{ulem}

\crefname{assumption}{Assumption}{Assumption}

\crefname{example}{Example}{Example}

\crefname{definition}{Definition}{Definition}

\makeatother
\crefformat{equation}{(#2#1#3)}

\numberwithin{equation}{section}

\contourlength{0.8pt}

\renewcommand{\underline}[1]{%
  \uline{\phantom{#1}}%
  \llap{\contour{white}{#1}}%
}

\definecolor{OliveGreen}{rgb}{0,0.6,0}
\definecolor{JaumeBlue}{rgb}{0,0,0.6}

\usepackage{babel}
  \providecommand{\conjecturename}{Conjecture}
  \providecommand{\corollaryname}{Corollary}
  \providecommand{\lemmaname}{Lemma}
  \providecommand{\propositionname}{Proposition}
  \providecommand{\remarkname}{Remark}
\providecommand{\theoremname}{Theorem}

\global\long\def\im{\text{Im}}

\global\long\def\di{\partial}

\global\long\def\Wout{W_{\text{out}}}

\DeclareMathOperator\vecop{vec}

\DeclareMathOperator\GL{GL}
\DeclareMathOperator\Tr{Tr}
\DeclareMathOperator\SPD{SPD}
\DeclareMathOperator\Sym{Sym}
\DeclareMathOperator\diag{diag}

\DeclareMathOperator\grad{grad}

\title{Geometric Dyson Brownian Motions and the Free Log-Normal Limit for a Non-Square Gaussian Matrix Product
} 

\author{
  Mufan Li
  \thanks{
    University of Waterloo and Vector Institute, 
    \texttt{mufan.li@uwaterloo.ca}
  }
  \and 
  Jaume de Dios Pont 
  \thanks{
    New York University, 
    \texttt{jdedios@nyu.edu}
  }
  \and 
  Mihai Nica 
  \thanks{
    University of Guelph and Vector Institute, 
    \texttt{nicam@uoguelph.ca}
  }
  \and 
  Daniel M. Roy
  \thanks{
    Google DeepMind, University of Toronto, and Vector Institute, 
    \texttt{daniel.roy@utoronto.ca}
  }
}

\date{}

\begin{document}

\maketitle

\vspace{-1em}
\begin{center}
\normalfont\itshape Dedicated to the memory of Tatsuhiro Onodera
\end{center}
\vspace{1em}

\begin{abstract}
We study the squared singular value spectrum of a non-square product of independent real Gaussian matrices, equivalently the feature covariance spectrum of a deep linear neural network at initialization.
Starting from the fixed-$m$ covariance diffusion previously obtained in the proportional depth-width limit, we record an equivalent matrix realization, describe its affine invariance, and derive the interacting diffusion satisfied by its eigenvalues.
We then take a second limit, sending $m\to\infty$ on the accelerated spectral clock $\tau=mt$, which corresponds in this sequential construction to the relation $dm/n\to\bar\tau$.
We establish convergence of the empirical spectral measure path to a deterministic mean-field limit and derive a closed Burgers equation for its $T$-transform.
Together with the proportional depth-width limit, these results give a rigorous sequential route from the deep non-square Gaussian product to the free log-normal limit of its feature covariance spectrum; for more general initial laws, the transform yields a free multiplicative convolution form.
We further analyze the support of the free log-normal law, give a fixed point iteration for numerical evaluation and a formal Marchenko--Pastur approximation at small time, and use the limiting spectrum to predict the risk in a toy random feature model.
\end{abstract}

\section{Introduction}

The recent success of deep neural networks has motivated the study of products of large random matrices \citep{hanin2019products,hanin2019finite}. 
In particular, a basic but not well understood example is the deep linear network 
\begin{equation}
    h_1 = \frac{1}{\sqrt{n_0}} W_0 x \,, \quad
    h_{\ell+1} = \frac{1}{\sqrt{n}} W_\ell h_\ell \,, \quad \ell\in[d-1],
\end{equation}
where $W_0\in\mathbb{R}^{n\times n_0}$, $W_\ell\in\mathbb{R}^{n\times n}$ for $\ell\in[d-1]$, and $x\in\mathbb{R}^{n_0\times m}$ is the input data set with $m$ points.
Here we can interpret $h_\ell \in \mathbb{R}^{n\times m}$ as a hidden layer of features of width $n$, and the whole neural network has depth $d$. 
An output layer may be fitted on top of the final features $h_d$ for a downstream task. 
In this work, we consider the network at initialization, i.e., all the weight matrices have entries $W_{\ell,ij} \overset{\mathrm{i.i.d.}}{\sim} \mathcal{N}(0,1)$.

While this setting reduces to a product of random matrices, it is still helpful to view the network as a recursion in the layer index. 
For example, when conditioned on the previous layer $h_\ell$, the next layer is Gaussian with an explicit covariance structure 
\begin{equation}
    \vecop( h_{\ell+1} ) \, | \, h_\ell 
    \overset{d}{=} 
    \mathcal{N}\left( 0 , \Phi_\ell \otimes I_n \right) \,, 
\end{equation}
where $\vecop:\mathbb{R}^{n\times m} \to \mathbb{R}^{nm}$ flattens the matrix into a column vector by stacking the columns from left to right, $\Phi_\ell = \frac{1}{n} h_\ell^\top h_\ell$ is the feature covariance matrix, and $\otimes$ denotes the Kronecker product. 

Furthermore, if we fit the output layer $\Wout$ using the random features $h_d$, then a central object in the analysis of generalization error is the spectrum of the feature covariance $\Phi_d = \frac{1}{n} h_d^\top h_d$ \cite{louart2018random,mei2022generalization,schroder2023deterministic,bosch2023precise,ghorbani2021linearized,xiao2022precise}.
The spectral viewpoint gives a simple way to connect the limiting law to downstream statistical performance.
In a random-feature model, the ridge prediction error is a linear spectral statistic of $\Phi_d$, so the limiting spectral law yields an explicit risk prediction.
Ultimately, we expect the spectrum of $\Phi_d$ to play an important role in studying trained neural networks as well.

Motivated by this long-term goal, we study the spectrum of $\Phi_d$ at initialization through a two-stage sequential limiting procedure.
The first stage is the fixed-$m$ proportional depth-width limit established in \citep{li2022neural}.
For each fixed number of data points $m$ and input dimension $n_0 \geq m$, that result sends $d,n\to\infty$ with
\begin{equation}
    \frac{d}{n}\to T>0
\end{equation}
and shows that the continuous interpolation of the covariance chain converges to a covariance diffusion.
We take this finite-$m$ diffusion as the starting point of the present analysis.
Writing $t := \ell/n \in [0,T]$ for depth time, we denote the resulting covariance process by $\Phi_t$.
The second limit is a large-$m$ spectral limit.
As will follow from the eigenvalue SDE in \cref{thm:geom_dyson}, the eigenvalue drift contains $m-1$ interaction terms, so the empirical spectral distribution evolves on the accelerated clock
$\tau := mt = \ell m/n$.
At spectral time $\tau$, the finite-$m$ covariance diffusion is observed at depth time $t=\tau/m$.

Within this two-stage construction, terminal spectral time $\bar\tau$ is represented at the discrete level by the relation
\begin{equation}
    \frac{dm}{n}\to \bar\tau .
\end{equation}
This identifies the clock of the sequential limit; it is not a direct triangular-array theorem in which $d,n,m$ diverge simultaneously.
Accordingly, if $\lambda_{i,t}$ denotes the $i$th eigenvalue of $\Phi_t$, we define the spectral-time particles and empirical spectral distribution by
\begin{equation}
\label{eq:intro-empirical-measure-timechanged}
    \mu_{i,\tau}
    :=
    \lambda_{i,\tau/m}
    =
    \lambda_i(\Phi_{\tau/m}),
    \qquad
    \rho_\tau^{(m)}
    :=
    \frac{1}{m}\sum_{i=1}^m
    \delta_{\mu_{i,\tau}} \,,
    \qquad
    \tau \in [0,\bar\tau] \,.
\end{equation}
Equivalently, $\mu_{i,\tau}$ is the $i$th eigenvalue of the finite-$m$ covariance diffusion $\Phi_{\tau/m}$ observed at depth time $t=\tau/m$.

\subsection{Main Results}
\label{subsec:main_results}

We begin with the finite-$m$ covariance diffusion obtained in \citet{li2022neural}. Under the first scaling $d,n \to \infty$ with $d/n \to T>0$ and $m,n_0$ fixed, the continuous-time interpolation of the covariance chain converges to the solution of the stochastic differential equation (SDE)
\begin{equation}
\label{eq:cov_sde}
    d\Phi_t = \Sigma^{1/2}(\Phi_t) \, dB_t \,,
    \quad
    \Phi_0 = \frac{1}{n_0} x^\top x \,, 
\end{equation}
where we treat the unique upper triangular entries of $\Phi$ as a vector in $\mathbb{R}^{m(m+1)/2}$, and the diffusion matrix has entries $\Sigma(\Phi)_{\alpha\beta,\gamma\delta} = \Phi^{\alpha\gamma}\Phi^{\beta\delta} + \Phi^{\alpha\delta}\Phi^{\beta\gamma}$. 
Our analysis starts from this diffusion. In \cref{thm:lin_cov_sde}, we record and prove the equivalent matrix representation $d\Phi_t=\Phi_t^{1/2}\,dS_t\,\Phi_t^{1/2}$, where $S_t$ is a symmetric Brownian motion. We then develop its affine-invariant consequences, project the diffusion to its eigenvalues, and study the sequential large-$m$ limit.

Unless explicitly stated otherwise, throughout the paper all finite-$m$ covariance processes are initialized in the positive definite cone:
\begin{equation}
    \Phi_0 \in \SPD(m) .
\end{equation}
In the neural-network parametrization $\Phi_0 = n_0^{-1}x^\top x$, this is equivalent to assuming that $x \in \mathbb{R}^{n_0 \times m}$ has full column rank $m$, and hence $m \leq n_0$. 
Degenerate finite-$m$ initial covariances are not treated in this paper.
\smallskip

Within this finite-$m$ diffusion, the equivalent matrix form reveals an affine-invariant symmetry (\cref{lm:markov_chain_affine_invar}).
In particular, if $\Phi_0 = QQ^\top$, then the solution started from $\Phi_0$ is equal in law to $Q \Phi_t^{(I_m)} Q^\top$, where $\Phi_t^{(I_m)}$ denotes the solution started from $I_m$; see \cref{rm:init_identity}.
This symmetry also allows us to show that $\Sigma^{-1}$ is the affine-invariant (Riemannian) metric for symmetric positive definite matrices (\cref{prop:sigma_affine_invar}) 
\begin{equation}
\label{eq:intro_affine_metric}
    g_\Phi(A,B) = \frac{1}{2} \Tr(\Phi^{-1} A \Phi^{-1} B) \,, 
    \quad 
    \Phi \in \SPD(m) \text{ and } A,B \in \Sym(m) \,. 
\end{equation}
Consequently, $\Phi_t$ also admits an information-geometric interpretation as Brownian motion associated with the dual connection; see \cref{rm:dual_brownian} and \cref{sec:appendix-dual-brownian}.

Projecting the covariance diffusion to its ordered eigenvalues gives the following multiplicative interacting particle system, a geometric variant of Dyson Brownian motion \citep{dyson1962brownian}.

\begin{restatable}[Geometric Dyson Brownian Motions]{theorem}{geomdbm}
\label{thm:geom_dyson}
Let $\Phi_t$ solve the covariance SDE of \cref{thm:lin_cov_sde} with $\Phi_0 \in \SPD(m)$, and let $\lambda_{1,t} \leq \lambda_{2,t} \leq \cdots \leq \lambda_{m,t}$ be the ordered eigenvalues of $\Phi_t$. Then the eigenvalues satisfy, locally on $(0,\infty)$,
\begin{equation}
\label{eq:gdbm_sde}
    d\lambda_{i,t}
    =
    \sqrt{2}\,\lambda_{i,t}\,dB_{i,t}
    + \sum_{\substack{j=1 \\ j\neq i}}^m
        \frac{\lambda_{i,t}\lambda_{j,t}}
             {\lambda_{i,t}-\lambda_{j,t}}\,dt,
    \qquad
    i = 1,2,\ldots,m .
\end{equation}
where $B_1,\ldots,B_m$ are independent standard Brownian motions.
\end{restatable}

Following \citet{ipsen2017mayWigner}, we call this a \emph{geometric} variant because it modifies geometric Brownian motions by adding an interaction drift.
A Doob-transform interpretation depends on the noise normalization; see \cref{subsec:log_dyson_calc}.
We can additionally calculate the law of the determinant ratio
\begin{equation}
    \log \frac{\det \Phi_t}{\det \Phi_0}
    \sim
    \mathcal{N}\left(
        -\frac{m(m+1)}{2}t \,,\, 2mt
    \right).
\end{equation}

The finite dimensional radial process admits an exact identification with the ellipsoid diffusion of \citet{norris1986brownian}, but the connection is subtle and is not apparent in the covariance coordinates in which our process arises. This finite dimensional identification does not provide the random product origin of the covariance diffusion or the large-$m$ spectral analysis below. We give the precise correspondence in \cref{subsec:log_dyson_calc}.

We next study the large-$m$ empirical spectrum of this covariance diffusion using the $T$-transform \cite[Section~16.3]{potters2020first}. For a probability measure $\rho$ on $\mathbb{R}_{\geq 0}$ and $z\in \mathbb{C} \setminus \mathbb{R}_{\geq 0}$, define
\begin{equation}
\label{eq:t_transform}
    G(z) 
    = \int_{\mathbb{R}} \frac{x}{z-x} \, d\rho(x) 
    = z g(z) - 1 \,, 
\end{equation}
where $g(z) = \int_{\mathbb{R}} \frac{1}{z-x} \, d\rho(x)$ is its Cauchy--Stieltjes transform.

Using the accelerated spectral clock and time-changed empirical measure in \cref{eq:intro-empirical-measure-timechanged}, we obtain the following deterministic global limit for the unscaled covariance spectrum.
\begin{restatable}[Sequential Mean-Field Limit and Burgers Equation]{theorem}{burgersfp}
\label{thm:burgers_fixed_point}
\label{thm:burgers-fixed-point}
Let $\rho^{(m)}_\tau := \frac{1}{m}\sum_{i=1}^m\delta_{\mu_{i,\tau}}$ be the empirical spectral distribution of the time-changed covariance $\Phi_{\tau/m}$. Assume that $\rho_0^{(m)} \xrightarrow{d} \rho_0$ as $m\to\infty$ and $\sup_m\langle x,\rho_0^{(m)}\rangle<\infty$.
Then $(\rho_\tau^{(m)})_{\tau\in[0,\bar\tau]}$ converges in probability in $C([0,\bar\tau],\mathcal{P}(\mathbb{R}_{\geq 0}))$ to a deterministic continuous path $(\rho_\tau)_{\tau\in[0,\bar\tau]}$. 

For each $z\in\mathbb{C}\setminus\mathbb{R}_{\geq 0}$, the $T$-transform $G_\tau(z):=\int x(z-x)^{-1}\rho_\tau(dx)$ satisfies
\begin{equation}
\label{eq:burgers_type_pde}
    \partial_\tau G_\tau(z)=-zG_\tau(z)\partial_zG_\tau(z),
    \qquad
    G_0(z)=\int \frac{x}{z-x}\,\rho_0(dx).
\end{equation}
\end{restatable}

For positive point-mass initial limits, \citet[Theorem~6]{maleckiPerez2022universality} prove the corresponding geometric-flow path limit. The theorem above treats arbitrary weak initial limits under the stated first-moment condition and proves uniqueness of the resulting path; \cref{subsec:related_work} gives the normalization and convergence-mode comparison. We then use the explicit $T$-transform equation to derive the inverse transform and free multiplicative convolution descriptions below.

Solving the Burgers equation by characteristics gives a particularly simple inverse transform description.

\begin{restatable}[Free Log-Normal Form]{corollary}{fixedpointflow}
\label{cor:fixed-point-free-flow}
When $\rho_0=\delta_1$, the limiting $T$-transform from \cref{thm:burgers_fixed_point} is the physical solution of
\begin{equation}
\label{eq:fixed_point}
    G_\tau(z)=\frac{1}{ze^{-\tau G_\tau(z)}-1},
    \qquad z\in\mathbb{C}\setminus\mathbb{R}_{\geq 0}.
\end{equation}
Denote the corresponding law by $\nu_\tau$. If, more generally, $\rho_0$ is compactly supported on $\mathbb{R}_{\geq 0}$ and has nonzero first moment, then $\rho_\tau=\rho_0\boxtimes\nu_\tau$.
\end{restatable}

Here ``physical solution'' means the branch corresponding to the $T$-transform of a probability measure, equivalently the branch with $G_\tau(z)\sim 1/z$ as $z\to\infty$.

After translating transform conventions and applying the mean-one dilation, the inverse-transform evolution and free multiplicative convolution description in \cref{cor:fixed-point-free-flow} agree with the subordination formulas of \citet[Propositions~2.4--2.5 and Example~4.12]{zhong2015freeMultiplicativeNormal}. In our setting, these identities are derived from the Burgers flow of the covariance diffusion, thereby connecting the free-probability description to the sequential limit of the deep non-square Gaussian matrix product.

\begin{figure}[t]
\centering
\includegraphics[width=0.8\linewidth]{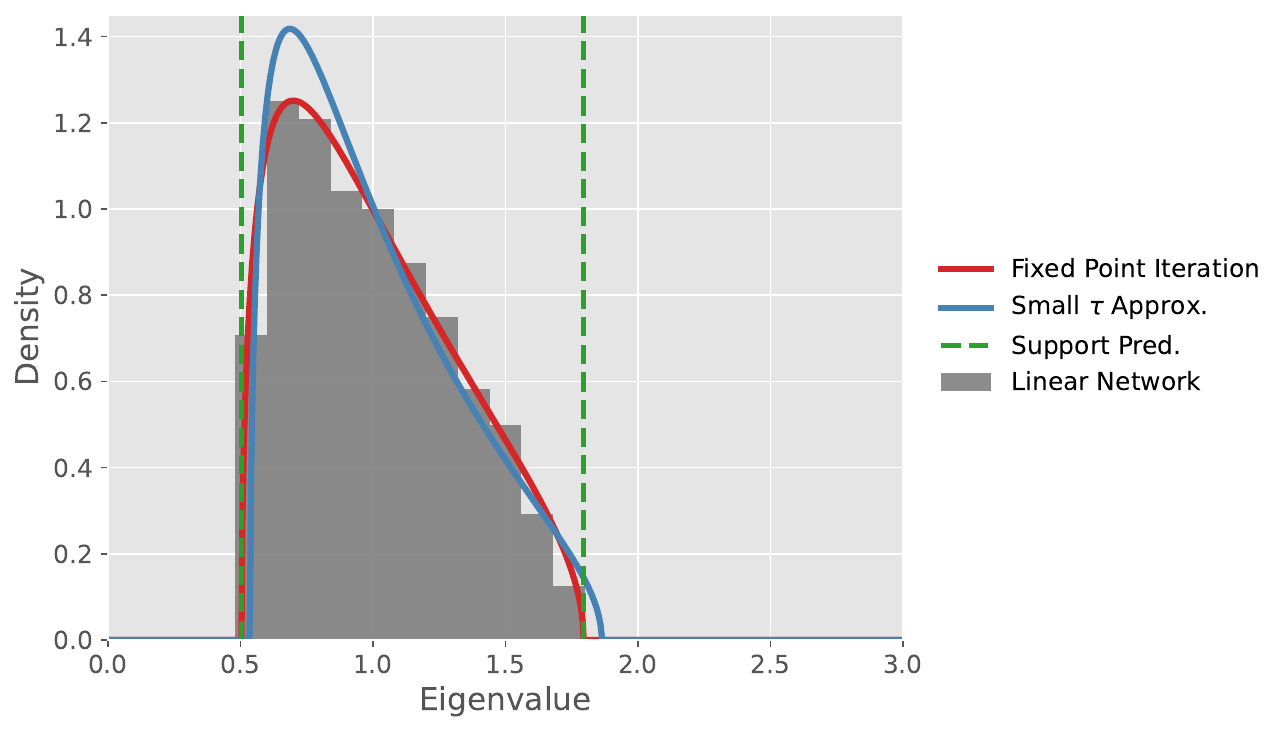}
\caption{Numerical density plot for the free log-normal limiting spectral law in \cref{cor:fixed-point-free-flow} at $\tau = 0.1$, computed using the fixed point iteration described in \cref{subsec:fixed-point-iteration} and compared with the formal approximation at small time in \eqref{eq:small_t_approx}. The approximation is an $x^{-1}$-reweighted Marchenko--Pastur density with parameters $\sigma^2=1+\tau$ and $c=\tau/(1+\tau)$. The simulation uses network width $n=10^4$, depth $d=5$, matrix size $m=200$, and input $n_0=m$, $x=\sqrt m I_m$.}
\label{fig:intro_density}
\end{figure}

The law $\nu_\tau$ is the \emph{free log-normal} \cite{potters2020first}, arising as the limit of a large number of multiplicative free convolutions \cite{bercovici1992levy}; see also \cite{biane1997free,kemp2013large} for related free Brownian motion and large-$N$ $\text{GL}_N$ Brownian motion perspectives.
Equivalently, $\nu_\tau=(\exp)_*(\mu^{\text{sc}}_\tau\boxplus u_{[-\tau,0]})$, where $\mu^{\text{sc}}_\tau$ is the centered semicircle law of variance $\tau$ and $u_{[-\tau,0]}$ is uniform on $[-\tau,0]$; this is the dilation by $e^{-\tau/2}$ of the law in \cite[Theorem~1.1]{auer2026freePositive}; see also \cite{auerVoit2025explicit}.

The support of the free positive multiplicative Brownian law was computed by \citet[Section~4.2]{biane1997segal}; see also \citet[Proposition~4.13]{zhong2015freeMultiplicativeNormal}. After the mean-one dilation $x\mapsto e^{-\tau/2}x$, writing $K_\tau(w)=e^{\tau w}\left(1+\frac{1}{w}\right)$ gives the equivalent endpoint formula
\begin{equation}
    \operatorname{supp}\nu_\tau = [K_\tau(w_-), K_\tau(w_+)] \,, 
    \qquad 
    w_\pm = \frac{-1 \pm \sqrt{1+4/\tau}}{2} \,. 
\end{equation}
Moreover, for the free log-normal law, a formal first-order small-time expansion (\cref{subsec:small_time_expansion}) leads to the approximation
\begin{equation}
\label{eq:small_t_approx}
    \nu_\tau(x)
    \approx \frac{ \sqrt{ (x-{\lambda}_-) ({\lambda}_+ - x) } }{ 2\pi \tau x^2 } \, ,
    \quad
    x \in [{\lambda}_-, {\lambda}_+] \,.
\end{equation}
where the corresponding edge locations are
${\lambda}_\pm = \left(\sqrt{1+\tau} \pm \sqrt{\tau}\right)^2$. 
This density is the $x^{-1}$-reweighted version of the scaled Marchenko--Pastur law in the covariance-matrix convention with parameters $\sigma^2=1+\tau$ and $c=\tau/(1+\tau)$.

The rest of the article is organized as follows. 
In \cref{sec:cov_sde_affine_invar}, we record the covariance SDE in matrix form and its affine-invariance properties. 
In \cref{sec:geometric-dyson-brownian-motions}, we project this diffusion to eigenvalues and derive the geometric Dyson Brownian motion. 
In \cref{sec:dynamical-wigner}, we take the large-$m$ mean-field limit, derive the Burgers equation for the $T$-transform, and analyze the support of the free log-normal law, an approximation at small time, and a toy random feature model.
We conclude in \cref{sec:discussion-conclusion} with open problems and directions for future work.

\subsection{Related Work}
\label{subsec:related_work}

Classical infinite width limits of neural networks, beginning with \citet{neal1995bayesian}, were extended to deep architectures, training dynamics, and generalization \cite{lee2018deep,du2019gradient,allen2019convergence,zou2020gradient,chizat2019lazy,lee2019wide,yang2019scaling,yang2020tensor,arora2019exact,chen2021how,ji2019polylogarithmic,ba2019generalization,bartlett2021deep}.
Quantitative work treats finite width approximation and universality under broader weight distributions \cite{basteri2024quantitative,bordino2024nonasymptotic,trevisan2023wide,apollonio2024normal,favaro2025quantitative,celli2025entropic,balasubramanian2024gaussian,balasubramanian2026finite,celli2026wide,giovagnini2026universality}, while mean field limits provide a complementary feature learning regime \cite{RVE18,cb18,SS18, MeiE7665, featurelearning,yang2022tensor,bordelon2022self}.

Within proportional infinite-depth-and-width limits, \citet{hanin2019products, hanin2019finite} treat the single data point case $m=1$, while \citet{li2022neural} derive the feature covariance SDE for fixed finite $m$. See also \cite{li2021future,seleznova2022neural,lidifferential,noci2024shaped} for related proportional limits and \cite{hayou2021stable,hayou2023commutative,hayou2023width,bordelondepthwise,yangtensorvi} for infinite-depth residual networks.

Joint growth of the number of inputs with depth and width was studied by \citet{hanin2023bayesian} in a Bayesian setting with $dm/n$ held fixed; follow-up work treats nonlinear models and Wishart process approaches \cite{hanin2024bayesian,bassetti2024proportional,hanin2026bayesian}.

At finite $m$, singular value SDEs arise for Wishart processes \citep{bru1989diffusions}, while \citet{graczyk2013multidimensional} treat broader matrix SDEs whose coefficients act spectrally, and related Dyson Brownian motions arise from Riemannian submersions \citep{menon2023siegel}.
At the stochastic top edge, \citet{assiotisMirsajjadi2026gln}, building on \citet{assiotisMirsajjadi2024isde}, study a complex $\beta=2$ counterpart with the same multiplicative interaction kernel. Their limit is a stochastic edge system, and its stochastic Burgers type equation concerns the logarithmic derivative of a characteristic entire function rather than the deterministic global empirical $T$-transform considered here.

At the global level, \citet[Section~16.3]{potters2020first} formally derive the same $T$-transform Burgers equation and free log-normal law from a different multiplicative recursion. For the real geometric matrix flow, \citet[Theorem~6]{maleckiPerez2022universality} prove empirical path convergence when the initial empirical measures converge to a positive point mass, characterizing the limit by a weak equation and moment recursion. Their general compactness theorem \cite[Theorem~1]{maleckiPerez2022universality} assumes a uniform eighth-moment bound and yields full convergence conditional on uniqueness. In the real $\beta=1$, zero-drift specialization and after setting their time parameter to $\tau/2$, their particle SDE and weak equation agree with ours, while their Cauchy-transform evolution \cite[Section~4.4]{maleckiPerez2022universality} specializes to \eqref{eq:burgers_type_pde}. \citet[Theorem~3.1]{auerVoit2025explicit} prove almost-sure weak convergence at each fixed time for deterministic starting configurations whose exponentiated empirical laws converge in every moment to a law satisfying an explicit moment-growth condition, and identify the limit by free multiplicative convolution. Their proof also gives locally uniform convergence of each empirical moment; the free log-normal formula is developed further in \citet{auer2026freePositive}.

The present result is distinct at both theorem and model levels. For the real covariance flow, \cref{thm:burgers_fixed_point} gives convergence in probability of the empirical path on every compact time interval for arbitrary weak initial limits under uniform first-moment control, allows mass at zero, and proves uniqueness of the limiting measure-valued path through negative-real-axis characteristics of its Burgers transform. For compactly supported initial laws with nonzero first moment, \cref{cor:fixed-point-free-flow} identifies the free multiplicative convolution description. Auer and Voit offer complementary almost-sure fixed-time convergence and an explicit convolution limit for possibly unbounded moment-controlled initial laws. In the sequential regime studied here, our theorem starts from the covariance diffusion obtained by \citet{li2022neural} as the proportional depth-width limit of a deep non-square Gaussian product and completes the model-to-law route to the global spectral flow and, in the stated compact-support regime, its free multiplicative convolution form.

Other spectral regimes include fixed-depth covariance kernels with $m,n\to\infty$ \cite{pennington2017nonlinear,louart2018random,peche2019note,fan2020spectra,wang2024nonlinear}, together with related universality results \cite{hanin2019products,hanin2019finite,pandit2024universality}, and products with $d \propto n=n_0=m$ stabilized by taking the $1/d$-th root \cite{liu2018lyapunov,liu2024phase,gorin2022gaussian,ahn2022fluctuations,ahn2022extremal,hanin2021non}.
In the latter product regime the limit is independent of the ratio between $d$ and $n$ \cite{hanin2021non}, whereas the present regime depends on $dm/n$; see also the related physics literature \cite{akemann2012universal,akemann2014universal,akemann2015recent,akemann2019integrable,akemann2014universal2}.

\section{The Covariance SDE and Affine-Invariance}
\label{sec:cov_sde_affine_invar}

We start by recalling the finite-width covariance chain for linear networks. 
The finite-$m$ covariance diffusion itself comes from the proportional depth-width limit of \citet{li2022neural}; here we also record the equivalent matrix form because it exposes the affine-invariant structure used later.
We begin by recalling the notation introduced earlier. 
Let $x \in \mathbb{R}^{n_0 \times m}$ have columns $[x^1,x^2,\cdots,x^m]$, each a separate input data point, and recall the hidden layers as
\begin{equation}
    h^\alpha_{\ell+1} = \frac{1}{\sqrt{n}} W_\ell h^\alpha_\ell
    = \left( \frac{1}{\sqrt n}W_\ell \cdots \frac{1}{\sqrt n}W_1 \right)\frac{1}{\sqrt{n_0}}W_0x^\alpha \,.
\end{equation}
where all the weight entries are i.i.d. $\mathcal{N}(0,1)$, and we will generalize to the non-uniform width case later. 

This allows us to write $h_\ell \in \mathbb{R}^{n\times m}$ as the collection of column vectors $[h^1_\ell, h^2_\ell, \cdots, h^m_\ell]$. 
Consequently, the feature covariance matrices are defined by 
\begin{equation}
    \Phi_\ell = \frac{1}{n} h_\ell^\top h_\ell = \frac{1}{n} \left[ \langle h^\alpha_\ell, h^\beta_\ell \rangle \right]_{\alpha,\beta=1}^m \,, 
    \quad 
    \Phi_0 = \frac{1}{n_0}x^\top x = \frac{1}{n_0}\left[\langle x^\alpha,x^\beta\rangle\right]_{\alpha,\beta=1}^m.
\end{equation}

\smallskip
\noindent
Throughout the sequel, unless explicitly stated otherwise, we work under the standing nondegeneracy assumption $\Phi_0 \in \SPD(m)$. In the neural-network parametrization above, this is equivalent to $x$ having full column rank.
\smallskip

Let $S_t$ denote the symmetric Brownian motion on $\Sym(m)$ normalized by
\begin{equation}
\label{eq:symmetric_brownian}
    d\langle S^{\alpha\beta},S^{\gamma\delta}\rangle_t
    =
    \left(
        \delta_{\alpha\gamma}\delta_{\beta\delta}
        +
        \delta_{\alpha\delta}\delta_{\beta\gamma}
    \right)dt,
    \qquad
    \alpha,\beta,\gamma,\delta\in[m].
\end{equation}
Equivalently, we can construct $S_t = \frac{1}{\sqrt{2}} ( W_t + W_t^\top )$, where $W_t$ has i.i.d. standard Brownian-motion entries.
\smallskip

\begin{theorem}
[Matrix Representation of the Covariance Diffusion]
\label{thm:lin_cov_sde}
Let $(\Phi_t)_{t\in[0,T]}$ denote the fixed-$m$ covariance diffusion obtained for the Gaussian linear network in \citet{li2022neural}, with initial condition $\Phi_0=n_0^{-1}x^\top x$. Then $\Phi_t$ admits the matrix representation\footnote{We would like to thank Tianze Jiang for showing us this matrix form of the covariance SDE.}
\begin{equation}
\label{eq:cov_sde_matrix}
    d\Phi_t
    =
    \Phi_t^{1/2}\,dS_t\,\Phi_t^{1/2} \,,
    \qquad
    \Phi_0 = \frac{1}{n_0}x^\top x \,,
\end{equation}
where $S_t$ is the symmetric Brownian motion introduced above, and $\Phi_t^{1/2}$ denotes the positive definite square root of $\Phi_t$. In upper-triangular coordinates, this representation is equivalent to the covariance SDE in \eqref{eq:cov_sde} obtained in \citet{li2022neural}.
\end{theorem}

\begin{proof}
For completeness, we verify the representation directly from the discrete Gaussian covariance chain.
Let $\mathcal F_\ell=\sigma(h_0,\ldots,h_\ell)$ where $h_0 = x$. Conditionally on $\mathcal F_\ell$, the rows of
$h_{\ell+1}$ are i.i.d. Gaussian vectors with covariance $\Phi_\ell$. Hence, if
$z_\ell\in\mathbb R^{n\times m}$ has i.i.d. $\mathcal N(0,1)$ entries and is independent of
$\mathcal F_\ell$, we may realize the transition as
\begin{equation}
    h_{\ell+1}\mid \mathcal F_\ell
    \overset d=
    z_\ell\Phi_\ell^{1/2}.
\end{equation}
Therefore
\begin{equation}
    \Phi_{\ell+1}\mid\mathcal F_\ell
    \overset d=
    \Phi_\ell^{1/2}
    \left(
        \frac{1}{n}z_\ell^\top z_\ell
    \right)
    \Phi_\ell^{1/2}.
\end{equation}
Equivalently, defining the centered Wishart fluctuation
\begin{equation}
    \xi_\ell
    :=
    \frac{1}{\sqrt n}
    \left(
        z_\ell^\top z_\ell-nI_m
    \right),
\end{equation}
we have
\begin{equation}
\label{eq:wishart_euler_update}
    \Phi_{\ell+1} 
    |\mathcal F_\ell
    \overset{d}{=}
    \Phi_\ell
    +
    \frac{1}{\sqrt n}
    \Phi_\ell^{1/2}\xi_\ell\Phi_\ell^{1/2}
\end{equation}
in conditional distribution.

The matrices $\xi_\ell$ are i.i.d., centered, and have the covariance of one unit increment
of the symmetric Brownian motion $S$. Indeed, for every $A,B\in\Sym(m)$,
\begin{equation}
\begin{aligned}
    \mathbb E\left[
        \Tr(A\xi_\ell)
        \Tr(B\xi_\ell)
    \right]
    =
    \mathbb E\left[
        \left(z^\top A z-\Tr A\right)
        \left(z^\top B z-\Tr B\right)
    \right]  
    =
    2\Tr(AB),
\end{aligned}
\end{equation}
where $z\sim\mathcal N(0,I_m)$. Equivalently, in coordinates,
\begin{equation}
    \mathbb E\left[
        \xi_\ell^{\alpha\beta}
        \xi_\ell^{\gamma\delta}
    \right]
    =
    \delta_{\alpha\gamma}\delta_{\beta\delta}
    +
    \delta_{\alpha\delta}\delta_{\beta\gamma}.
\end{equation}
Since $m$ is fixed, the entries of $\xi_\ell$ have uniformly bounded moments of all orders.
Thus the martingale array $S_t^{(n)}
    :=
    \sum_{\ell< nt}
    \frac{1}{\sqrt n}\xi_\ell$
satisfies
\begin{equation}
    (S_t^{(n)})_{t\in[0,T]}
    \xrightarrow{d}
    (S_t)_{t\in[0,T]},
\end{equation}
where the limit is the symmetric Brownian motion introduced above in \cref{eq:symmetric_brownian}.
Thus, the recursion \eqref{eq:wishart_euler_update} is precisely the Euler scheme
\begin{equation}
    \Phi_{\ell+1}
    =
    \Phi_{\ell}
    +
    \Phi_{\ell}^{1/2}
    \left(
	        S_{(\ell+1)/n}^{(n)} - S_{\ell/n}^{(n)}
    \right)
    \Phi_{\ell}^{1/2} \,.
\end{equation}
By the standard Markov-chain diffusion approximation (see e.g. \cite[Appendix A]{li2022neural}), equivalently the weak convergence of Euler schemes after localization on $\SPD(m)$, 
	the continuous-time interpolation satisfies
    \begin{equation}
        ( \Phi_{\lfloor nt\rfloor} )_{t\in[0,T]}
        \xrightarrow{d}
        (\Phi_t)_{t\in[0,T]},
    \end{equation}
    where the limit solves
\begin{equation}
	    d\Phi_t
	    =
	    \Phi_t^{1/2}\,dS_t\,\Phi_t^{1/2} \,,
    \qquad
    \Phi_0=\frac{1}{n_0}x^\top x \,.
\end{equation}

It remains only to identify the upper-triangular coordinate covariance. From the cross variation of
	$S_t$, we have
\begin{equation}
\begin{aligned}
	    d\langle \Phi^{\alpha\beta},\Phi^{\gamma\delta}\rangle_t
	    &=
	    \sum_{p,q,r,s}
	    (\Phi_t^{1/2})^{\alpha p}
	    (\Phi_t^{1/2})^{q\beta}
	    (\Phi_t^{1/2})^{\gamma r}
	    (\Phi_t^{1/2})^{s\delta}
	    d\langle S^{pq},S^{rs}\rangle_t  \\
	    &=
	    \left(
	        \Phi_t^{\alpha\gamma}\Phi_t^{\beta\delta}
	        +
	        \Phi_t^{\alpha\delta}\Phi_t^{\beta\gamma}
	    \right)dt.
\end{aligned}
\end{equation}
Thus, when the unique upper-triangular entries are used as coordinates, the limiting SDE has
desired diffusion matrix 
$\Sigma(\Phi)_{\alpha\leq\beta,\gamma\leq\delta}
    =
    \Phi^{\alpha\gamma}\Phi^{\beta\delta}
    +
    \Phi^{\alpha\delta}\Phi^{\beta\gamma}$. 
\end{proof}

\begin{remark}
\label{rm:non-unif-width}
The same proof extends to non-uniform widths by viewing the update for each layer $\ell$ 
\begin{equation}
    \Phi_{\ell+1} | \mathcal{F}_\ell
    \overset{d}{=}
    \Phi_\ell
    +
    \frac{1}{\sqrt{n_{\ell+1}}}
    \Phi_\ell^{1/2}\xi_\ell\Phi_\ell^{1/2} 
\end{equation}
as a discretization of the limiting SDE for the transition $\Phi_\ell \to \Phi_{\ell+1}$ with step size $\frac{1}{n_{\ell+1}}$. 
Therefore, if all step sizes tend to zero and their sum converges to the same layer time, the same limit follows. 
More precisely, we would take $d, n_1, n_2, \cdots, n_d \to \infty$ with $\max_{1\leq \ell\leq d}\frac{1}{n_\ell}\to 0$ and
$\sum_{\ell=1}^d \frac{1}{n_\ell} \to T$, which was also the approach taken in \cite{hanin2019products,hanin2019finite}.
\end{remark}

\subsection{Affine-Invariance and Geometric Structure}

Whenever an SDE with a non-identity diffusion coefficient $\Sigma$ arises, it is natural to speculate whether or not $\Sigma^{-1}$ corresponds to a Riemannian metric. 
In this case, we provide an affirmative answer, where $\Sigma^{-1}$ is the affine-invariant metric for the manifold of symmetric positive definite matrices $M = \SPD(m)$. 
We start by establishing a key symmetry structure. 
\begin{lemma}[Affine-Invariance]
\label{lm:markov_chain_affine_invar}
Assume $n\ge m$, and let $P=P_n:\SPD(m)\to\SPD(m)$ be the random Markov chain update taking $\Phi_\ell$ to $\Phi_{\ell+1}$.
Then, for every $\Phi \in \SPD(m)$ and every $A \in \GL(m,\mathbb{R})$,
\begin{equation}
    A P(\Phi) A^\top \overset{d}{=} P( A \Phi A^\top ) \,.
\end{equation}
Consequently, the Markov semigroup, equivalently the transition kernel, of the covariance SDE in \cref{thm:lin_cov_sde} is affine-invariant.
\end{lemma}

\begin{proof}

By the conditional Wishart representation of one layer, the Markov update from $\Phi \in \SPD(m)$ can be realized as
\begin{equation}
    P(\Phi)
    \overset{d}{=}
    \frac{1}{n}\sum_{i=1}^n g_i g_i^\top,
    \qquad
    g_i \overset{\mathrm{i.i.d.}}{\sim} \mathcal{N}(0,\Phi).
\end{equation}
Because $n\ge m$, the sample covariance is positive definite almost surely, so $P_n$ indeed maps $\SPD(m)$ into $\SPD(m)$.

Now fix $A \in \GL(m,\mathbb{R})$. Since $g_i \sim \mathcal{N}(0,\Phi)$, we have $Ag_i \sim \mathcal{N}(0,A\Phi A^\top)$. Therefore,
\begin{equation}
\begin{aligned}
    P(A\Phi A^\top)
    \overset{d}{=}
    \frac{1}{n}\sum_{i=1}^n (Ag_i)(Ag_i)^\top 
    =
    A\left(\frac{1}{n}\sum_{i=1}^n g_i g_i^\top\right)A^\top 
    \overset{d}{=}
    A P(\Phi) A^\top .
\end{aligned}
\end{equation}
This proves the one-step affine-invariance identity. 
Iterating the one-step identity gives affine invariance of the finite-width Markov chain transition kernels. 
Passing to the diffusion limit in \cref{thm:lin_cov_sde} gives affine invariance of the limiting covariance SDE semigroup.

\end{proof}

\begin{remark}
\label{rm:init_identity}
Let $\Phi_t^{(V)}$ denote the solution of the covariance SDE in \cref{thm:lin_cov_sde} started from $V \in \SPD(m)$.
If $V = QQ^\top$, then \cref{lm:markov_chain_affine_invar} implies 
\begin{equation}
    \Phi_t^{(V)} \overset{d}{=} Q \Phi_t^{(I_m)} Q^\top \, ,
\end{equation}
where $\Phi_t^{(I_m)}$ solves the same covariance SDE with initial condition $I_m$.
Thus these identities may be reduced to the case $\Phi_0=I_m$.
\end{remark}

This reduction will be used later at the infinitesimal level: after rotating to the instantaneous eigenbasis, the diagonal noise gives the multiplicative Brownian terms, while the off-diagonal noise gives the repulsion term through the second variation of eigenvalues.

Let $g_\Phi(A,B)$ define the Riemannian inner product at $\Phi \in \SPD(m)$ with respect to the affine-invariant metric, i.e. for all $A,B \in \Sym(m) = T_\Phi \SPD(m) = \Sym(m)$, we have 
\begin{equation}
    g_\Phi(A,B) = \frac{1}{2} \Tr( \Phi^{-1} A \Phi^{-1} B ) \,. 
\end{equation}
Next, we identify $\Sigma^{-1}$ with the affine-invariant metric $g_\Phi$.

\begin{proposition}
[$\Sigma^{-1}$ is the Affine-Invariant Metric]
\label{prop:sigma_affine_invar}
For all $\Phi \in \SPD(m)$ and $A,B \in \Sym(m)$, we have
\begin{equation}
    \vecop(A)^\top \Sigma(\Phi)^{-1} \vecop(B)
    = \frac{1}{2}\Tr(A\Phi^{-1}B\Phi^{-1})
    = g_\Phi(A,B),
\end{equation}
where $\vecop:\Sym(m)\to\mathbb{R}^{m(m+1)/2}$ flattens the upper triangular entries of a symmetric matrix into a vector.
\end{proposition}

\begin{proof}
For $A\in\Sym(m)$, let $F_A(\Phi):=\Tr(A\Phi)$. Since $dF_A[V]=\Tr(AV)$, the covectors $dF_A$ span the cotangent space as $A$ ranges over $\Sym(m)$.

Using the matrix covariance SDE from \cref{thm:lin_cov_sde} and cyclicity of the trace, the differential of $F_A$ is
\begin{equation}
    dF_A(\Phi_t)
    = \Tr(A\,d\Phi_t)
    = \Tr\!\left(\Phi_t^{1/2}A\Phi_t^{1/2}\,dS_t\right).
\end{equation}
By the normalization of the symmetric Brownian motion $S_t$, this gives
\begin{equation}
\begin{aligned}
    d\langle F_A(\Phi_t),F_B(\Phi_t)\rangle
    &= 2\Tr\!\left(\Phi_t^{1/2}A\Phi_t B\Phi_t^{1/2}\right)dt \\
    &= 2\Tr(A\Phi_t B\Phi_t)dt.
\end{aligned}
\end{equation}
Thus the covariance tensor of the diffusion acts on the covectors $dF_A,dF_B$ by $2\Tr(A\Phi B\Phi)$.

We now compute the cometric induced by $g_\Phi$. The covector $dF_A$ is represented by the tangent vector $2\Phi A\Phi$, because for every $V\in\Sym(m)$,
\begin{equation}
\begin{aligned}
    g_\Phi(2\Phi A\Phi,V)
    &= \frac12\Tr\!\left(\Phi^{-1}(2\Phi A\Phi)\Phi^{-1}V\right) \\
    &= \Tr(AV)
     = dF_A[V].
\end{aligned}
\end{equation}
Therefore
\begin{equation}
\begin{aligned}
    g_\Phi^{-1}(dF_A,dF_B)
    &= g_\Phi(2\Phi A\Phi,2\Phi B\Phi) \\
    &= 2\Tr(A\Phi B\Phi).
\end{aligned}
\end{equation}
The covariance tensor of the SDE therefore agrees with the affine-invariant cometric $g_\Phi^{-1}$ on the spanning family of covectors $dF_A$, and hence on the full cotangent space.

By \cref{thm:lin_cov_sde}, $\Sigma(\Phi)$ is the coordinate matrix of the covariance tensor in the upper-triangular coordinates $\{\Phi^{\alpha\beta}\}_{\alpha\leq\beta}$. Since we have identified this covariance tensor with the cometric $g_\Phi^{-1}$, the inverse matrix $\Sigma(\Phi)^{-1}$ is the coordinate matrix of $g_\Phi$ in the same coordinates. Therefore, for tangent vectors $A,B\in\Sym(m)$,
\begin{equation}
    \vecop(A)^\top \Sigma(\Phi)^{-1}\vecop(B)
    =
    g_\Phi(A,B)
    =
    \frac12\Tr(A\Phi^{-1}B\Phi^{-1}).
\end{equation}
This proves the claim.

\end{proof}

\begin{remark}[Dual Brownian Interpretation]
\label{rm:dual_brownian}
The covariance diffusion is not the Levi--Civita Brownian motion associated with the affine-invariant metric in \cref{prop:sigma_affine_invar}.
However, the covariance tensor identified in \cref{prop:sigma_affine_invar} is the affine-invariant cometric.
As a result, the It\^o generator can be identified as (one half of) the dual-connection Laplacian for the information geometry of centered Gaussian measures, and $\Phi_t$ may be viewed as Brownian motion associated with the dual connection.
The calculation is given in \cref{sec:appendix-dual-brownian}.
\end{remark}

\section{Geometric Dyson Brownian Motions}
\label{sec:geometric-dyson-brownian-motions}

In this section, we project the affine-invariant covariance diffusion to its eigenvalues and prove \cref{thm:geom_dyson}. The calculation is the geometric analogue of the standard derivation of Dyson Brownian motion, with the affine-invariant noise replacing additive matrix Brownian noise.

For a smooth curve $(\Phi_t)_{t \geq 0}$ in $\SPD(m)$ with simple spectrum, the first and second Hadamard variation formulae \cite[eq. 1.73 and 1.74]{tao2012topics} give the following identities for the eigenvalues $\lambda_i$ of $\Phi_t$:
\begin{equation}
    \dot \lambda_i = u_i^\top \dot \Phi_t u_i \,, \quad
    \ddot \lambda_i = u_i^\top \ddot \Phi_t u_i
    + 2 \sum_{j \neq i} \frac{ |u_j^\top \dot \Phi_t u_i |^2 }{ \lambda_i - \lambda_j } \,,
\end{equation}
where $\{u_i\}_{i \in [m]}$ are the orthonormal eigenvectors of $\Phi_t$.

Therefore if we choose $\Phi_t = \Phi + t \Psi$, this gives us the directional derivatives of
\begin{equation}
\label{eq:eigen_deriv}
    \langle \nabla \lambda_i( \Phi ), \Psi \rangle = u_i^\top \Psi u_i \,, \quad 
    \nabla^2 \lambda_i( \Phi ) [ \Psi, \Psi ] = 2 \sum_{j \neq i} \frac{ |u_j^\top \Psi u_i |^2 }{ \lambda_i - \lambda_j } \,. 
\end{equation}

We now prove \cref{thm:geom_dyson}.

\geomdbm*

\begin{proof}
By \cref{lem:instantaneous-simplicity-no-collision}, applied with $\varepsilon=1/k$ and then intersected over $k\geq 1$, the spectrum is simple for every positive time almost surely. It is therefore enough to compute the dynamics locally on the open set of simple-spectrum matrices.

At a time with simple spectrum, choose the orthonormal eigenbasis of $\Phi$. By \cref{lm:markov_chain_affine_invar}, and in particular by orthogonal invariance, the local eigenvalue generator may be computed after this orthogonal change of basis. Since orthogonal conjugation does not change eigenvalues, we compute at $\Phi=\diag(\lambda_1,\ldots,\lambda_m)$.

In this basis, \cref{thm:lin_cov_sde} gives $d\Phi_{ij}=\sqrt{\lambda_i\lambda_j}\,dS_{ij}$. Hence the quadratic variations needed below are
\begin{equation}
\begin{aligned}
    d\langle \Phi_{ii},\Phi_{jj}\rangle_t
    &= 2\lambda_i^2\delta_{ij}\,dt, \\
    d\langle \Phi_{ij},\Phi_{ij}\rangle_t
    &= \lambda_i\lambda_j\,dt,
    \qquad i\neq j .
\end{aligned}
\end{equation}
At the same diagonal matrix, \cref{eq:eigen_deriv} reduces to
\begin{equation}
    \langle \nabla\lambda_i(\Phi),\Psi\rangle=\Psi_{ii},
    \qquad
    \nabla^2\lambda_i(\Phi)[\Psi,\Psi]
    =2\sum_{j\neq i}\frac{\Psi_{ij}^2}{\lambda_i-\lambda_j} .
\end{equation}
Applying It\^o's formula, using that the matrix entries have no It\^o drift, gives
\begin{equation}
    d\lambda_{i,t}
    =
    d\Phi_{ii}
    +
    \sum_{j\neq i}
    \frac{d\langle \Phi_{ij},\Phi_{ij}\rangle_t}{\lambda_i-\lambda_j} .
\end{equation}
Since $d\Phi_{ii}=\lambda_i\,dS_{ii}$ and $d\langle S_{ii},S_{ii}\rangle_t=2\,dt$, we may write
\begin{equation}
    d\Phi_{ii}=\sqrt{2}\lambda_i\,dB_{i,t} .
\end{equation}
The relation $d\langle \Phi_{ii},\Phi_{jj}\rangle_t=0$ for $i\neq j$ shows that $B_1,\ldots,B_m$ are independent. Substituting the off-diagonal quadratic variations gives
\begin{equation}
    d\lambda_{i,t}
    =
    \sqrt{2}\lambda_{i,t}\,dB_{i,t}
    +
    \sum_{j\neq i}
    \frac{\lambda_{i,t}\lambda_{j,t}}
         {\lambda_{i,t}-\lambda_{j,t}}\,dt,
    \qquad i=1,\ldots,m .
\end{equation}
This proves the stated SDE locally on $(0,\infty)$.
\end{proof}

\subsection{Log Transform and Related Calculations}
\label{subsec:log_dyson_calc}

Because the diffusion coefficient is $\sqrt{2}\lambda_{i,t}$, it is natural to set $\chi_{i,t}:=\log\lambda_{i,t}$. Applying It\^{o}'s lemma gives
\begin{equation}
\label{eq:log_geom_dyson}
    d\chi_{i,t}
    =
    \sqrt{2}\,dB_{i,t}
    + \sum_{j\neq i}
        \frac{1}{e^{\chi_{i,t}-\chi_{j,t}}-1}\,dt
    - dt \,.
\end{equation}

Let $r_{i,t}:=\chi_{i,t}/2$ be the log singular values. Using $(e^{2x}-1)^{-1}=\tfrac12(\coth x-1)$ in \cref{eq:log_geom_dyson}, we obtain
\begin{equation}
\label{eq:log-singular-value-gdbm}
    dr_{i,t}
    =
    \frac{1}{\sqrt{2}}\,dB_{i,t}
    + \frac{1}{4}\sum_{j\neq i}
        \coth(r_{i,t}-r_{j,t})\,dt
    - \frac{m+1}{4}\,dt \,.
\end{equation}
This is the case $a=1/2$, $b=c=0$, and $\mu=(m+1)/4$ of the geometric Dyson Brownian motions studied by \citet{ipsen2017mayWigner} in terms of finite time Lyapunov exponents.

The log singular value equation also makes precise a subtle connection with the radial ellipsoid diffusion of \citet{norris1986brownian}. Let $G_s$ be their right invariant Brownian motion on $\mathrm{GL}(m,\mathbb{R})$, with $G_0=I_m$, and set $\mathcal{G}_s:=G_s\Phi_0^{1/2}$ and $Y_s:=\mathcal{G}_s^{\mathsf T}\mathcal{G}_s$. If $\Lambda_i(Y_s)$ denotes the $i$th ordered eigenvalue of $Y_s$ and $\gamma_i(s):=\tfrac12\log\Lambda_i(Y_s)$, then the covariance diffusion may be coupled with $Y_s$ so that
\begin{equation}
\label{eq:norris-covariance-normalization}
    \Phi_t = e^{-(m+1)t/2}Y_{t/2} \qquad r_{i,t} = \gamma_i(t/2)-\frac{m+1}{4}t.
\end{equation}
Although compactly expressed by \cref{eq:norris-covariance-normalization}, this identification is far from apparent. Establishing it requires recognizing the right invariant matrix Brownian motion behind the covariance realization, passing from squared singular values to radial log coordinates, and matching the matrix orientation, stochastic calculus, time, and drift conventions. It therefore gives an exact identification of the two finite dimensional processes once each has been constructed; it does not derive the covariance diffusion, its random product origin, or the large-$m$ analysis from \citet{norris1986brownian}.

When $|\chi_{i,t}-\chi_{j,t}|$ is small, the denominator satisfies $e^{\chi_{i,t}-\chi_{j,t}}-1\approx \chi_{i,t}-\chi_{j,t}$, so the log eigenvalues have a local repulsion similar to that of Dyson Brownian motion. 

The $\beta=2$ Dyson Brownian motion can be interpreted as a Doob $h$-transform, with the ordinary Vandermonde as a positive harmonic function on the Weyl chamber \cite{bloemendal2010doob}. It is therefore natural to compare the geometric variant with this construction.
We start by identifying an analogue of the Vandermonde determinant and its drift-corrected version:
\begin{equation}
\label{eq:drift-corrected-vandermonde}
    \Delta_m(\chi) = \prod_{i < j} ( e^{-\chi_i} - e^{-\chi_j} )
    \qquad
    h(\chi) = \Delta_m(\chi) \exp\left( -\sum_{i=1}^m \chi_i \right) \,,
\end{equation}
Since the log eigenvalues $\chi_i$ are increasing, $\Delta_m>0$ whenever the spectrum is simple. 

This form is useful because the interaction part of the drift can be written as
\begin{equation}
\label{eq:modified-vandermonde-log-derivative}
    \di_i \log \Delta_m(\chi) = \sum_{j\neq i} \frac{1}{ e^{\chi_i - \chi_j} - 1 } \,, 
\end{equation}
Since $\di_i\log h=\di_i\log\Delta_m-1$, the full drift in \cref{eq:log_geom_dyson} is $\nabla\log h$. The identity in \cref{lm:vandermonde_almost_harmonic} and the analogous calculation for $h$ give
\begin{equation}
\label{eq:drift-corrected-vandermonde-eigenvalues}
    \Delta \Delta_m(\chi) = \frac{m(m-1)(2m-1)}{6} \Delta_m(\chi)
    \qquad
    \Delta h(\chi) = \frac{m(m+1)(2m+1)}{6} h(\chi) \,.
\end{equation}
Thus, after the usual scalar eigenvalue normalization, replacing the Brownian coefficient $\sqrt{2}$ in \cref{eq:log_geom_dyson} by $1$ yields a Doob $h$-transform on the ordered chamber. Equivalently, this auxiliary diffusion is obtained by conditioning independent Brownian motions with ordered drifts $(-m,-m+1,\ldots,-1)$ never to collide; see \citet[Section~5.2]{bianeBougerolOConnell2005littelmann} and \cref{prop:unit-noise-conditioning}. With coefficient $\sqrt{2}$, the same transform produces drift $2\nabla\log h$, rather than the drift $\nabla\log h$ in \cref{eq:log_geom_dyson}, and therefore does not recover the process studied here.

Nevertheless, the log transform is still useful: we end this section with a determinant identity obtained from the symmetry of the log dynamics. 

\begin{corollary}
[Determinant Distribution]
\label{cor:det_dist}
If $\Phi_0\in\SPD(m)$ is deterministic and $\Phi_t$ evolves according to the covariance SDE of \cref{thm:lin_cov_sde}, then
\begin{equation}
    \log \frac{\det \Phi_t}{\det \Phi_0}
    \sim
    \mathcal{N}\left(
        -\frac{m(m+1)}{2}t \,,\, 2mt
    \right) \,.
\end{equation}
\end{corollary}

\begin{proof}

Summing \cref{eq:log_geom_dyson} over $i$ gives\footnote{Strictly speaking, if $\Phi_0$ does not have simple spectrum, one may apply the calculation on $[\varepsilon,t]$, using \cref{lem:instantaneous-simplicity-no-collision}, and then let $\varepsilon \downarrow 0$.}
\begin{equation}
    d \log \det \Phi_t
    =
    \sqrt{2} \sum_{i=1}^m dB_{i,t}
    +
    \sum_{i \neq j}
    \frac{1}{e^{\chi_{i,t}-\chi_{j,t}}-1} \, dt
    -
    m \, dt \,.
\end{equation}
Pairing the ordered terms $(i,j)$ and $(j,i)$, we have
\begin{equation}
    \frac{1}{e^{\chi_i-\chi_j}-1}
    +
    \frac{1}{e^{\chi_j-\chi_i}-1}
    =
    -1 \,.
\end{equation}
Thus the interaction drift contributes $-\binom{m}{2} \, dt$. Since $m^{-1/2}\sum_{i=1}^m B_i(t)$ is a standard Brownian motion, we obtain
\begin{equation}
    d \log \det \Phi_t
    =
    \sqrt{2m} \, dB_t
    -
    \frac{m(m+1)}{2} \, dt \,.
\end{equation}
Integrating from $0$ to $t$ gives
\begin{equation}
    \log \frac{\det \Phi_t}{\det \Phi_0}
    =
    \sqrt{2m}\,B_t
    -
    \frac{m(m+1)}{2}t \,.
\end{equation}
Thus the stated Gaussian law follows.

\end{proof}

\section{Mean-Field Limit and the Free Log-Normal Law}
\label{sec:dynamical-wigner}

This section studies the second, large-$m$ step of the sequential limit. The compactness and subsequential limit argument for \cref{thm:burgers_fixed_point} is given in \cref{sec:appendix-meanfield-gdbm}, while the transform calculation appears below. We then solve the Burgers equation by characteristics to obtain \cref{cor:fixed-point-free-flow}. \cref{lem:meanfield-uniqueness} provides the real axis uniqueness argument and the inverse branch identity used to justify the comparison of transform germs.

We use the $T$-transform $G$ defined in \cref{eq:t_transform}. Since $G(z)=zg(z)-1$ for probability measures, pointwise convergence $G_{\mu_n}(z)\to G_\mu(z)$ on $\mathbb C\setminus\mathbb R_{\geq 0}$ is equivalent to convergence of the Cauchy--Stieltjes transforms, and hence to $\mu_n\xrightarrow{d}\mu$ for $\mu_n,\mu\in\mathcal P(\mathbb R_{\geq 0})$.

A similar calculation gives the inversion formula. Let $z=a+ib$ with $b>0$, and let $P_b(t)=\frac{b}{\pi(t^2+b^2)}$ be the Poisson kernel. Writing $\rho^\sharp(dx):=x\,\rho(dx)$, we have
\begin{equation}
    \label{eq:t_transform_poisson_identity}
    \im\,G(a+ib)
    = \int \frac{-b x}{(x-a)^2+b^2}\,d\rho(x)
    = -\pi\,(P_b*\rho^\sharp)(a)\,.
\end{equation}
Consequently, at any $a>0$ at which $\rho$ has a density and $a$ is a Lebesgue point of that density,
\begin{equation}
    \label{eq:inversion_formula}
    \rho(a) = \lim_{b\to 0^+} \im \frac{ -G(a+ib) }{ \pi a } 
    = \lim_{b\to 0^+} \im \frac{ G(a-ib) - G(a+ib) }{ 2\pi a } \,, 
\end{equation}
Note the extra factor of $a$ in the denominator differs from the usual inversion formula. 

\paragraph{Time-Change Convention.}
We now use the large-$m$ clock convention introduced above: $\tau=mt=\ell m/n$ denotes the accelerated spectral time, so the finite-$m$ covariance diffusion is evaluated at depth time $t=\tau/m$. If $\lambda_{i,t}$ is the $i$th eigenvalue of $\Phi_t$, set
\begin{equation}
\label{eq:timechanged-empirical-measure-section4}
    \mu_{i,\tau}:=\lambda_{i,\tau/m},
    \qquad
    \rho_\tau^{(m)}
    := \frac{1}{m}\sum_{i=1}^m \delta_{\mu_{i,\tau}}.
\end{equation}
Under this time change, the processes $\widehat B_{i,\tau}:=\sqrt m\,B_{i,\tau/m}$ are standard Brownian motions in the $\tau$-clock. After this Brownian rescaling, we drop the hats and write them again as $B_{i,\tau}$. The geometric Dyson Brownian motion for the spectral-time particles becomes
\begin{equation}
\label{eq:gdbm_time_change}
    d\mu_{i,\tau}
    =
    \sqrt{\frac{2}{m}}\,\mu_{i,\tau}\,dB_{i,\tau}
    + \frac{1}{m}\sum_{j\neq i}
        \frac{\mu_{i,\tau}\mu_{j,\tau}}
             {\mu_{i,\tau}-\mu_{j,\tau}}\,d\tau \, .
\end{equation}

The $T$-transform is natural here because the geometric interaction kernel contains the multiplicative factor $xy$. For $f_z(x)=x/(z-x)$, this factor closes the limiting weak equation in terms of $G_\tau(z)$ and $\partial_zG_\tau(z)$. This is the multiplicative analogue of the Stieltjes-transform closure in additive Dyson Brownian motion.

We now restate and prove the main mean-field limit result. 

\burgersfp*

In the main text we isolate the transform calculation that leads from the particle system to the limiting Burgers equation. The tightness, compact-containment estimates, identification of subsequential limits, and uniqueness of the limiting measure-valued path follow the standard scheme of \citet[Proposition 4.3.10 and Lemma 4.3.13]{anderson2010introduction} and are collected in \cref{sec:appendix-meanfield-gdbm} for the present geometric system.

\begin{proof}

We work on the fixed terminal spectral-time interval $[0,\bar\tau]$. 
We first record the weak equation suggested by the particle system \eqref{eq:gdbm_time_change}. 
Let
\begin{equation}
\label{eq:meanfield-test-class}
    \mathcal F_b
    :=
    \left\{
        f\in C^2(\mathbb{R}_{\geq 0})\cap C_b(\mathbb{R}_{\geq 0}) :
        \|x f'(x)\|_\infty + \|x^2 f''(x)\|_\infty < \infty
    \right\}.
\end{equation}
For $f\in\mathcal F_b$, define
\begin{equation}
\label{eq:meanfield-Hf}
    H_f(x,y)
    :=
    \begin{cases}
        \displaystyle
        \frac{xy\bigl(f'(x)-f'(y)\bigr)}{x-y}, & x\ne y, \\[1em]
        x^2 f''(x), & x=y.
    \end{cases}
\end{equation}
The diagonal value is the continuous extension of the off-diagonal expression; the boundedness and continuity estimates are proved in \cref{sec:appendix-meanfield-gdbm}.

Applying It\^o's formula to $\langle f,\rho_\tau^{(m)}\rangle$ and symmetrizing the interaction drift gives the finite-$m$ identity
\begin{equation}
\label{eq:finite-m-weak-main}
\begin{aligned}
    \langle f,\rho_\tau^{(m)}\rangle
    &=
    \langle f,\rho_0^{(m)}\rangle
    +
    M_\tau^{(m,f)}
    \\
    &\quad
    +
    \frac{1}{2}\int_0^\tau
    \iint_{(\mathbb{R}_{\geq 0})^2}
    H_f(x,y)\,\rho_s^{(m)}(dx)\rho_s^{(m)}(dy)\,ds
    \\
    &\quad
    +
    \frac{1}{2m}\int_0^\tau
    \langle x^2 f''(x),\rho_s^{(m)}\rangle\,ds,
\end{aligned}
\end{equation}
where $M_\tau^{(m,f)}$ is the It\^o martingale. The precise martingale estimate, compact containment, and passage to the limit in \eqref{eq:finite-m-weak-main} are given in \cref{lem:meanfield-technical}. Therefore the sequence $(\rho_\tau^{(m)})_{\tau\in[0,\bar\tau]}$ is tight in $C([0,\bar\tau],\mathcal P(\mathbb R_{\geq0}))$, and every subsequential limit $(\rho_\tau)_{\tau\in[0,\bar\tau]}$ is continuous and, for every fixed $f\in\mathcal F_b$, almost surely satisfies, for all $\tau\in[0,\bar\tau]$,
\begin{equation}
\label{eq:limit-weak}
    \langle f,\rho_\tau\rangle
    =
    \langle f,\rho_0\rangle
    +
    \frac{1}{2}
    \int_0^\tau
    \iint_{(\mathbb{R}_{\geq 0})^2}
    H_f(x,y)\,\rho_s(dx)\rho_s(dy)\,ds.
\end{equation}

Let $(\rho_\tau)_{\tau\in[0,\bar\tau]}$ be one such subsequential limit. We now identify the equation satisfied by its $T$-transform; uniqueness will then follow from \cref{lem:meanfield-uniqueness}.

Now fix $z\in\mathbb{C}\setminus\mathbb{R}_{\geq 0}$ and define
\begin{equation}
    f_z(x):=\frac{x}{z-x} \,. 
\end{equation}
Since
\begin{equation}
    f_z'(x)=\frac{z}{(z-x)^2} \,,
    \qquad
    f_z''(x)=\frac{2z}{(z-x)^3} \,,
\end{equation}
the real and imaginary parts of $f_z$ belong to $\mathcal F_b$. Define
\begin{equation}
    G_\tau(z):=\langle f_z,\rho_\tau\rangle
    =
    \int_{\mathbb{R}_{\geq 0}} \frac{x}{z-x}\,\rho_\tau(dx) \,. 
\end{equation}
Applying \eqref{eq:limit-weak} to the real and imaginary parts of $f_z$ and recombining gives
\begin{equation}
    G_\tau(z)
    =
    G_0(z)
    +
    \frac{1}{2}
    \int_0^\tau
    \iint_{(\mathbb{R}_{\geq 0})^2}
    H_{f_z}(x,y)\,\rho_s(dx)\rho_s(dy)\,ds \,. 
\end{equation}
A direct computation yields
\begin{equation}
    f_z'(x)-f_z'(y)
    =
    z\frac{(z-y)^2-(z-x)^2}{(z-x)^2(z-y)^2}
    =
    z\frac{(y-x)(x+y-2z)}{(z-x)^2(z-y)^2} \,,
\end{equation}
hence
\begin{equation}
    H_{f_z}(x,y)
    =
    \frac{xyz(2z-x-y)}{(z-x)^2(z-y)^2}
    =
    z\frac{x}{z-x}\frac{y}{(z-y)^2}
    +
    z\frac{y}{z-y}\frac{x}{(z-x)^2} \,. 
\end{equation}
By symmetry of the double integral,
\begin{equation}
    \frac{1}{2}
    \iint H_{f_z}(x,y)\,\rho_s(dx)\rho_s(dy)
    =
    z
    \left(
        \int \frac{x}{z-x}\,\rho_s(dx)
    \right)
    \left(
        \int \frac{y}{(z-y)^2}\,\rho_s(dy)
    \right) \,. 
\end{equation}
Using
\begin{equation}
    \partial_z G_s(z)
    =
    -
    \int_{\mathbb{R}_{\geq 0}} \frac{x}{(z-x)^2}\,\rho_s(dx) \,,
\end{equation}
we obtain
\begin{equation}
    G_\tau(z)
    =
    G_0(z)
    -
    \int_0^\tau z G_s(z)\,\partial_z G_s(z)\,ds \,. 
\end{equation}
Therefore
\begin{equation}
    \partial_\tau G_\tau(z) = - z G_\tau(z)\,\partial_z G_\tau(z) \,,
    \qquad
    G_0(z)=\int_{\mathbb{R}_{\geq 0}} \frac{x}{z-x}\,\rho_0(dx) \,. 
\end{equation}

Applying the subsequential-limit argument simultaneously to the real and imaginary parts of $f_z$ for $z$ in a countable dense subset of $\mathbb{C}\setminus\mathbb{R}_{\geq 0}$, and then using holomorphy in $z$, places this transform equation on a single probability-one event for every $z\in\mathbb{C}\setminus\mathbb{R}_{\geq 0}$. The uniqueness argument in \cref{lem:meanfield-uniqueness} then shows that every subsequential limit equals the same deterministic path. Therefore the full sequence converges in law to this deterministic path, which is equivalent to convergence in probability in $C([0,\bar\tau],\mathcal{P}(\mathbb{R}_{\geq 0}))$.

\end{proof}

\fixedpointflow*

\begin{proof}
To make the model-to-law connection explicit, we derive the transform identities directly from the Burgers equation in the mean-one covariance normalization.
Write the Burgers equation from \cref{thm:burgers_fixed_point} as
\begin{equation}
    \partial_\tau G_\tau(z)+zG_\tau(z)\partial_zG_\tau(z)=0.
\end{equation}
For a characteristic curve, take
\begin{equation}
    \dot Z_s=Z_sG_s(Z_s),
    \qquad
    Z_\tau=z.
\end{equation}
Then
\begin{equation}
    \frac{d}{ds}G_s(Z_s)=0.
\end{equation}
Writing $w=G_\tau(z)$, we have $G_s(Z_s)=w$ along the characteristic, and therefore $\dot Z_s=wZ_s$. Hence
\begin{equation}
    Z_0=ze^{-\tau w}=ze^{-\tau G_\tau(z)}.
\end{equation}
This gives the characteristic relation
\begin{equation}
    G_\tau(z)=G_0\!\left(ze^{-\tau G_\tau(z)}\right).
\end{equation}
\cref{lem:meanfield-uniqueness} rigorously justifies this characteristic identity on the negative real axis, selects the physical branch, and records the inverse branch identity used below.

When $\rho_0=\delta_1$, $G_0(z)=1/(z-1)$. Thus, for $u<0$, \cref{lem:meanfield-uniqueness} gives
\begin{equation}
    G_\tau(u)=\frac{1}{ue^{-\tau G_\tau(u)}-1}.
\end{equation}
Define
\begin{equation}
    F_\tau(z):=G_\tau(z)\bigl(ze^{-\tau G_\tau(z)}-1\bigr)-1.
\end{equation}
The function $F_\tau$ is holomorphic on $D=\mathbb C\setminus\mathbb R_{\geq 0}$ and vanishes on $(-\infty,0)$. By the identity theorem, $F_\tau\equiv0$ on $D$, and hence
\begin{equation}
    G_\tau(z)=\frac{1}{ze^{-\tau G_\tau(z)}-1},
    \qquad z\in D,
\end{equation}
which is \eqref{eq:fixed_point}.

Since the inverse branch for $\nu_\tau$ is \(K_{\nu_\tau}(w)=e^{\tau w}(1+1/w)\), the relation \(K_\mu(w)=(1+1/w)/S_\mu(w)\) gives \(S_{\nu_\tau}(w)=e^{-\tau w}\) as a germ near the origin. Here the \(S\)-transform is well-defined by the compact support of \(\nu_\tau\) shown in \cref{cor:fixed_point_support} below.

We now identify the general compactly supported solution. By the inverse-branch identity in \cref{lem:meanfield-uniqueness}, if \(K_\tau\) and \(K_0\) denote the inverse branches of \(G_\tau\) and \(G_0\) corresponding to \(z=\infty\), then
\begin{equation}
    K_\tau(w)=e^{\tau w}K_0(w).
\end{equation}
On the other hand, multiplicativity of the \(S\)-transform under free multiplicative convolution gives
\begin{equation}
\begin{aligned}
    K_{\rho_0\boxtimes\nu_\tau}(w)
    &= \frac{1+1/w}{S_{\rho_0\boxtimes\nu_\tau}(w)} \\
    &= \frac{1+1/w}{S_{\rho_0}(w)e^{-\tau w}} \\
    &= e^{\tau w}K_0(w)
     = K_\tau(w).
\end{aligned}
\end{equation}
Thus the \(T\)-transforms agree as inverse germs near \(w=0\). Since the \(T\)-transform determines the probability measure,
\begin{equation}
    \rho_\tau=\rho_0\boxtimes\nu_\tau.
\end{equation}
\end{proof}

\subsection{Support of the Free Log-Normal}

We now recover the known support formula in the mean-one normalization directly from the inverse branch in \cref{eq:fixed_point}; see \citet[Section~4.2]{biane1997segal} and \citet[Proposition~4.13]{zhong2015freeMultiplicativeNormal}. The real-branch argument in \cref{lem:support-branch-analysis} verifies that the two real critical values are the actual support endpoints.

\begin{corollary}[Support of the Free Log-Normal]
\label{cor:fixed_point_support}
Let $\nu_\tau$ be the free log-normal law from \cref{cor:fixed-point-free-flow}.
The law $\nu_\tau$ is supported on the single interval
\begin{equation}
\label{eq:exact_support}
    \operatorname{supp}\nu_\tau = [K_\tau(w_-), K_\tau(w_+)] \,, 
    \quad 
    K_\tau(w)=e^{\tau w}\left(1+\frac{1}{w}\right) \,, 
    \quad 
    w_\pm = \frac{-1 \pm \sqrt{1+4/\tau}}{2} \,. 
\end{equation}
\end{corollary}

\begin{proof}
Set $w=G_\tau(z)$ in \cref{eq:fixed_point}. Then the fixed-point equation gives the inverse relation
\begin{equation}
    z = K_\tau(w) := e^{\tau w}\left(1+\frac{1}{w}\right).
\end{equation}
By \cref{lem:support-branch-analysis}, the support of the physical branch is the interval whose endpoints are the two real critical values of $K_\tau$.
Differentiating gives
\begin{equation}
    K_\tau'(w)
    = e^{\tau w}\frac{\tau w^2+\tau w-1}{w^2}.
\end{equation}
Thus the critical points satisfy
\begin{equation}
    \tau w^2+\tau w-1=0,
\end{equation}
and hence
\begin{equation}
    w_\pm
    = \frac{-1\pm\sqrt{1+4/\tau}}{2}.
\end{equation}
Evaluating $K_\tau$ at these two points gives \cref{eq:exact_support}.
\end{proof}

\subsection{Formal Small-Time Expansion}
\label{subsec:small_time_expansion}

In the free log-normal case, we can also extract a formal small-time approximation to the density of $\nu_\tau$ directly from the fixed-point equation. Starting from
\begin{equation}
    G_\tau(z)=\frac{1}{ze^{-\tau G_\tau(z)}-1} 
\end{equation}
and making the first-order formal approximation
\begin{equation}
    e^{-\tau G_\tau(z)} \approx 1 - \tau G_\tau(z)
\end{equation}
gives us the quadratic equation 
\begin{equation}
\label{eq:formal_small_time_quadratic_closure}
    \tau z G_\tau(z)^2 - (z-1) G_\tau(z) + 1 = 0 \,. 
\end{equation}
Choosing the branch with $G_\tau(z) \sim 1/z$ as $|z| \to \infty$, this closure gives the formal approximation 
\begin{equation}
\label{eq:formal_small_time_transform}
    G_\tau(z)
    \approx \frac{z-1-\sqrt{(z-1)^2-4\tau z}}{2\tau z} \,. 
\end{equation}
The corresponding branch cut lies on the interval with endpoints
\begin{equation}
\label{eq:formal_small_time_edges}
    {\lambda}_\pm
    = \left(\sqrt{1+\tau} \pm \sqrt{\tau}\right)^2 \,. 
\end{equation}
Using the $T$-transform inversion formula \cref{eq:inversion_formula}, this formal calculation leads to the small-time approximation:
\begin{equation}
\label{eq:small_time_mp_density}
    \nu_\tau(x)
    \approx \frac{\sqrt{(x-{\lambda}_-)({\lambda}_+-x)}}{2\pi\tau x^2} \, ,
    \quad 
    x \in [{\lambda}_-, {\lambda}_+] \,, 
\end{equation}
and the extra factor of $x^{-1}$ in \cref{eq:small_time_mp_density} comes from the inversion formula $\nu_\tau(x)=-(\pi x)^{-1}\Im G_\tau(x+i0)$.

The density in \cref{eq:small_time_mp_density} is the $x^{-1}$-reweighted version of the scaled Marchenko--Pastur law in the covariance-matrix convention with parameters $\sigma^2=1+\tau$ and $c=\tau/(1+\tau)$.

\subsection{Fixed-Point Iteration}
\label{subsec:fixed-point-iteration}

For the free log-normal case, the fixed-point equation in \cref{cor:fixed-point-free-flow} provides a direct numerical realization of the
limiting \(T\)-transform. Fix \(\tau>0\) and \(z\in\mathbb{C}\setminus\mathbb{R}\), and define
\begin{equation}
T_{\tau,z}(w) := \frac{1}{z e^{-\tau w}-1}.
\end{equation}
We seek the fixed point \(w=T_{\tau,z}(w)\) corresponding to the analytic branch of \(G_\tau\),
characterized by
\begin{equation}
G_\tau(z)\sim \frac{1}{z}\quad \text{as } |z|\to\infty,
\qquad
\Im G_\tau(z)<0\quad \text{for } \Im z>0.
\end{equation}
A natural initialization is the \(\tau=0\) value
\begin{equation}
w^{(0)} = G_0(z)=\frac{1}{z-1},
\end{equation}
and we then iterate
\begin{equation}
w^{(k+1)}=(1-\theta)w^{(k)}+\theta\,T_{\tau,z}(w^{(k)}),
\qquad 0<\theta\le 1,
\end{equation}
until both
\begin{equation}
\label{eq:fixed-point-stopping-rule}
\left|w^{(k+1)}-w^{(k)}\right|
\le \varepsilon\left(1+\left|w^{(k+1)}\right|\right),
\qquad
\left|w^{(k+1)}-T_{\tau,z}\!\left(w^{(k+1)}\right)\right|
\le \varepsilon\left(1+\left|w^{(k+1)}\right|\right) \, .
\end{equation}
For $\Im z>0$, we additionally accept only an iterate with $\Im w^{(k+1)}<0$. The residual and half-plane checks are essential: a small relaxed update alone need not certify either the fixed-point equation or consistency with the transform branch. The half-plane condition is necessary but does not by itself select the desired fixed point; that branch is tracked through the initialization and continuation described below. At a fixed point, $T'_{\tau,z}(w)=\tau w(w+1)$, and this derivative equals one at the real support critical points. Thus the iteration becomes nearly noncontractive near the support edges as the regularization parameter tends to zero. Moderate damping can improve robustness away from this limiting degeneracy, but does not remove the edge slowdown.

This is the safeguarded fixed-point iteration used to compute the numerical density plot in \cref{fig:intro_density}. For that figure, we used $\tau=0.1$, $\eta=10^{-3}$, $\theta=0.20$, $\varepsilon=10^{-12}$, a cap of $10^4$ iterations, and a grid of $1600$ points on $[10^{-4},3]$. The first grid point was initialized with $G_0$, and each subsequent point was warm-started from the preceding solution. The displayed curve was masked outside the exact support to remove the small regularization tails produced by $\eta>0$.

Once \(G_\tau(x+i\eta)\) has been computed on a grid of \(x\)-values with a small \(\eta>0\), the
inversion formula \eqref{eq:inversion_formula} yields the regularized density
\begin{equation}
p_{\tau,\eta}(x):= -\frac{1}{\pi x}\Im G_\tau(x+i\eta).
\end{equation}
Using the support interval in \cref{eq:exact_support}, one may sweep across an \(x\)-grid covering
the support and use the converged value at the previous grid point as the initial guess
at the next point. 
For larger values of \(\tau\), or whenever the residual or half-plane check fails at the target time, the iteration can be further stabilized by continuation
in time: choose \(0=\tau_0<\tau_1<\cdots<\tau_L=\tau\), solve the fixed-point problem successively at
each \(\tau_r\), and warm-start the computation at \(\tau_{r+1}\) with the converged solution at
\(\tau_r\). 
We do not pursue a convergence analysis here, but these safeguards give a reproducible and
effective numerical approximation of the limiting spectral law.

\subsection{A Toy Random-Feature Model}
\label{subsec:theory-practice-generalization}

The spectral law above suggests a simple downstream prediction for the random-feature model based on the final features \(h_d\). 
We keep this comparison deliberately modest: it is still a fixed-design, spectrum-only toy model, whose purpose is to test whether the limiting spectral law gives a useful finite-dimensional prediction.
This viewpoint is in the same spirit as spectral analyses of random-feature and kernel ridge regression \cite{louart2018random,mei2022generalization,schroder2023deterministic,bosch2023precise,ghorbani2021linearized,xiao2022precise}, but here the feature covariance itself is produced by the depth evolution studied above.

Let
\begin{equation}
\label{eq:generalization-final-covariance}
    \Phi_d=\frac{1}{n}h_d^\top h_d
\end{equation}
be the final feature covariance. For a ridge parameter \(\gamma>0\), the fitted values on the \(m\) design points are
\begin{equation}
\label{eq:ridge-fitted-values}
    \widehat f_\gamma
    =
    \Phi_d(\Phi_d+\gamma I_m)^{-1}y .
\end{equation}
We use a simple feature-teacher label model:
\begin{equation}
\label{eq:feature-teacher-label-model}
\begin{aligned}
    y &= f^{(m)}+\varepsilon^{(m)},
    &
    f^{(m)} &= h_d^\top\beta, \\
    \beta &\sim \mathcal N\!\left(0,\frac{\sigma_1^2}{n}I_n\right),
    &
    \varepsilon^{(m)} &\sim \mathcal N(0,\sigma_\varepsilon^2I_m).
\end{aligned}
\end{equation}
The random variables above are independent conditional on \(h_d\), and the conditional law of \(f^{(m)}\) depends on \(h_d\) only through \(\Phi_d\):
\begin{equation}
\label{eq:feature-teacher-conditional-covariance}
    f^{(m)}\mid \Phi_d
    \sim
    \mathcal N\!\left(0,\sigma_1^2\Phi_d\right).
\end{equation}
Conditional on \(\Phi_d\), the risk is
\begin{equation}
\label{eq:finite-m-spectral-risk}
    \mathcal E_m(\gamma\mid\Phi_d)
    =
    \frac{1}{m}
    \Tr\!\left[
        \sigma_1^2\gamma^2\Phi_d(\Phi_d+\gamma I_m)^{-2}
        +
        \sigma_\varepsilon^2\Phi_d^2(\Phi_d+\gamma I_m)^{-2}
    \right].
\end{equation}
\begin{figure}[t]
\centering
\includegraphics[width=0.7\linewidth]{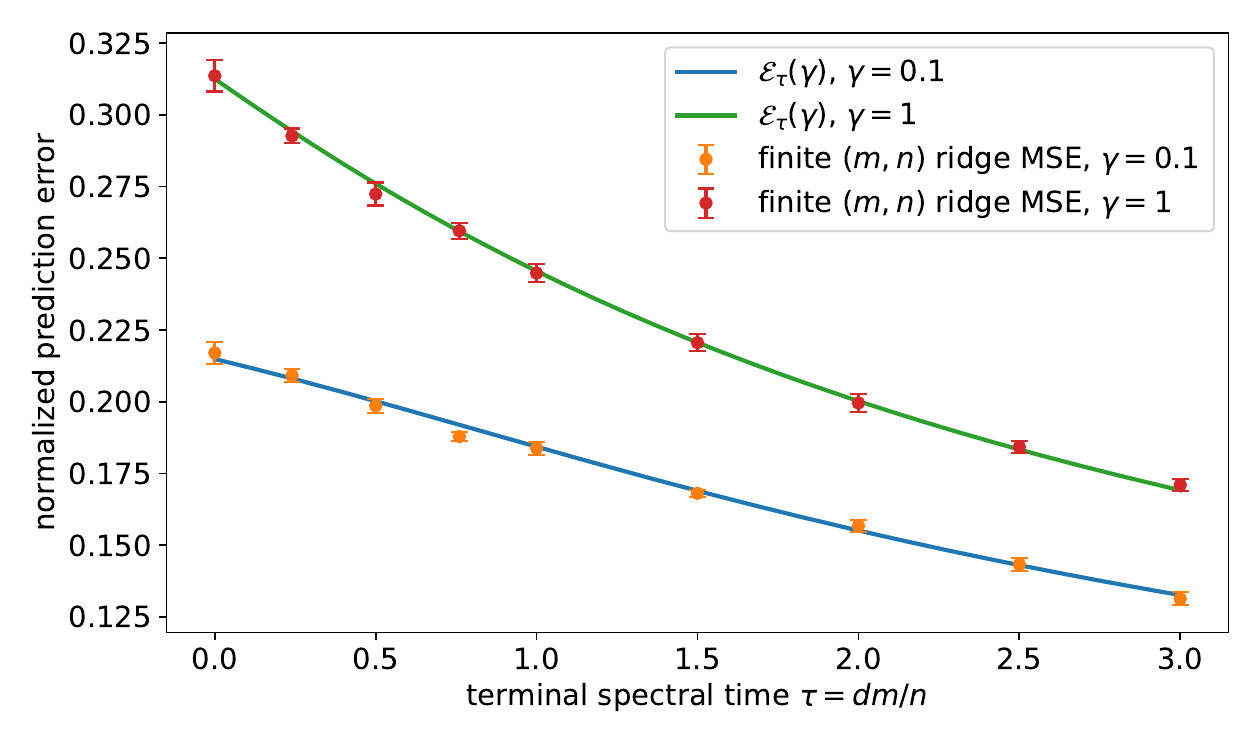}
\caption{Comparison between the limiting spectrum-based prediction \(\mathcal E_\tau(\gamma)\) and the finite-\((m,n)\) ridge prediction error in the feature-teacher toy model with $\rho_0=\delta_1$. The curves show the theoretical prediction computed from \(G_\tau\), and the markers show the empirical ridge mean-squared error from the covariance-chain simulation at matched values of \(\tau=dm/n\). Error bars show \(\pm 2\) standard errors. The figure uses \(\sigma_1^2=1\) and \(\sigma_\varepsilon^2=1/4\).}
\label{fig:theory-practice-generalization}
\end{figure}
The conditional Gaussian trace calculation and the transform identities used below are recorded in \cref{sec:appendix-toy-ridge-risk}. Thus the conditional error is a linear spectral statistic of \(\Phi_d\). If \(d\) is chosen so that \(dm/n\approx\tau\), the limiting prediction is
\begin{equation}
\label{eq:limiting-spectral-risk}
    \mathcal E_\tau(\gamma)
    =
    \int_{\mathbb R_{\geq 0}}
    \frac{\sigma_1^2\gamma^2x+\sigma_\varepsilon^2x^2}{(x+\gamma)^2}\,
    \rho_\tau(dx).
\end{equation}
In terms of the limiting \(T\)-transform,
\begin{equation}
\label{eq:limiting-risk-transform}
    \mathcal E_\tau(\gamma)
    =
    -\sigma_1^2\gamma^2\partial_zG_\tau(-\gamma)
    +
    \sigma_\varepsilon^2\bigl(-G_\tau(-\gamma)+\gamma\partial_zG_\tau(-\gamma)\bigr).
\end{equation}
When $\rho_0=\delta_1$, \(G_\tau\) is the physical branch from \cref{eq:fixed_point}; differentiating that fixed-point equation gives
\begin{equation}
\label{eq:identity-start-transform-derivative}
    \partial_zG_\tau(z)
    =
    \frac{G_\tau(z)^2e^{-\tau G_\tau(z)}}{\tau G_\tau(z)^2+\tau G_\tau(z)-1}.
\end{equation}
Equations \eqref{eq:fixed_point}, \eqref{eq:limiting-risk-transform}, and \eqref{eq:identity-start-transform-derivative} give the theoretical curves in \cref{fig:theory-practice-generalization}.

We compare this prediction with a finite covariance-chain simulation. For terminal spectral times \(\tau_j\), choose depths \(d_j\) with \(d_jm/n\approx\tau_j\), and let \(\Phi_{d_j}^{(b)}\) denote the covariance matrix in simulation batch \(b\). For each label replicate \(r\), we draw \(f_{b,r}^{(m)}\mid\Phi_{d_j}^{(b)}\sim\mathcal N(0,\sigma_1^2\Phi_{d_j}^{(b)})\), draw independent noise \(\varepsilon_{b,r}^{(m)}\sim\mathcal N(0,\sigma_\varepsilon^2I_m)\), and set \(y_{b,r}=f_{b,r}^{(m)}+\varepsilon_{b,r}^{(m)}\). With \(R\) independent label draws per batch, the empirical error is
\begin{equation}
\label{eq:empirical-finite-risk}
    \widehat{\mathcal E}_{m,n}(\tau_j,\gamma)
    =
    \frac{1}{BRm}
    \sum_{b=1}^B\sum_{r=1}^R
    \left\|
        \Phi_{d_j}^{(b)}
        \bigl(\Phi_{d_j}^{(b)}+\gamma I_m\bigr)^{-1}y_{b,r}
        -
        f_{b,r}^{(m)}
    \right\|_2^2 .
\end{equation}
The point of the comparison is to place \(\mathcal E_{\tau_j}(\gamma)\) and \(\widehat{\mathcal E}_{m,n}(\tau_j,\gamma)\) on the same spectral-time scale.

\cref{fig:theory-practice-generalization} shows that the limiting law gives an accurate finite-dimensional prediction in this setting. The limitation is equally important: the spectrum-only reduction works here because the conditional target covariance is a scalar function of \(\Phi_d\). For targets whose covariance is not a function of \(\Phi_d\), or for deterministic and structured targets, the eigenvector alignment with \(\Phi_d\) matters. A genuine out-of-sample generalization theory would also require train--test information not encoded by the empirical spectrum alone.

\section{Discussion} 
\label{sec:discussion-conclusion}

Together with the proportional depth-width limit of \citet{li2022neural}, our large-$m$ theorem provides, to our knowledge, the first rigorous route from a deep non-square Gaussian matrix product to the free log-normal limit of its feature covariance spectrum in the sequential regime considered here. The free log-normal law itself was studied previously in \cite{bercovici1992levy}; its Burgers equation appears formally in \cite{potters2020first}, and \citet[Theorem~6]{maleckiPerez2022universality} prove the corresponding geometric-flow empirical path limit for positive point-mass initial measures.

The random-feature model in \cref{subsec:theory-practice-generalization}, while intentionally a toy fixed-design calculation, demonstrates that the limiting spectral law can give quantitatively accurate predictions for ridge regression.
Extending this comparison to more realistic settings would require understanding how training dynamics affect the spectral evolution, as well as how target structure aligns with the eigenvectors of the feature covariance.

Several natural directions remain open. One is universality beyond Gaussian weights. The first limiting step uses the exact conditional Gaussian structure of the layerwise covariance chain. It remains to determine whether broader weight distributions converge to the same covariance diffusion, or at least to the same large-$m$ limiting spectral path.

A second direction is a direct joint limit in which the depth $d$, width $n$, and number of data points $m$ diverge simultaneously. Such a result would require approximation estimates for the finite-$d,n$ covariance recursion that remain uniform as $m$ grows. The simulations in \cref{fig:intro_density} suggest that the sequential law may remain informative at finite dimensions, but they do not establish equality of the sequential and joint limits.

A third direction, as briefly discussed above, is to study a more realistic setting, where the network allows for non-linearities, and the training dynamics are included. 
Both yield significant challenges. 
Non-linearities break the affine-invariance structure, which couples the eigenvalue and eigenvector dynamics and leads to a much more complicated system. 
Training dynamics break the independence structure of the random matrices, which was a key ingredient in the analysis of this work.

\section*{Acknowledgements}

The authors would like to thank 
Benjamin Budway, 
Sinho Chewi, 
Boris Hanin, 
Tianze Jiang, 
Govind Menon, 
Pierre Mergny, 
and Leonard Wong 
for insightful discussions and pointers on this topic. 
ML and JdDP would also like to thank the Princeton ML Theory Summer School for bringing the authors together for this project. 
ML acknowledges the support of the Natural Sciences and Engineering Research Council of Canada (NSERC), RGPIN-2026-05287.
DMR is supported by an NSERC Discovery Grant and a Canada CIFAR AI Chair at the Vector Institute.

\printbibliography[heading=bibintoc]

\appendix

\section{Calculation on the Vandermonde Determinant}
\label{sec:appendix-vandermonde}

Since $\frac{\di_x(e^{-x}-e^{-y})}{e^{-x}-e^{-y}}=\frac{1}{e^{x-y}-1}$, define the modified Vandermonde determinant by
\begin{equation}
\label{eq:modified-vandermonde}
    \Delta_m(\chi) 
    = \prod_{1 \leq i < j \leq m} 
    \left( e^{ - \chi_i} - e^{- \chi_j} \right) \,, 
\end{equation}
which satisfies 
\begin{equation}
\label{eq:modified-vandermonde-derivative}
    \di_i \Delta_m(\chi) 
    = \sum_{j \neq i} \frac{ \Delta_m(\chi) }{ e^{\chi_i - \chi_j} - 1 } \,. 
\end{equation}

However, it is not harmonic; indeed 
\begin{equation}
\label{eq:modified-vandermonde-second-derivative}
\begin{aligned}
    \di_{ii} \Delta_m(\chi) 
    &= 
        \di_i \sum_{j \neq i} 
        \frac{ \Delta_m(\chi) }{ e^{\chi_i - \chi_j} - 1 } 
        \\ 
    &= 
        \sum_{j \neq i} \sum_{ k \neq i }
        \frac{\Delta_m(\chi)}{ ( e^{\chi_i - \chi_j} - 1 ) ( e^{\chi_i - \chi_k} - 1 ) } 
        - \sum_{j \neq i} \frac{ \Delta_m(\chi) e^{\chi_i - \chi_j} }{
        ( e^{\chi_i - \chi_j} - 1 )^2 } 
        \\ 
    &= 
        \sum_{j \neq i} \sum_{ k \neq i }
        \frac{\Delta_m(\chi)}{ ( e^{\chi_i - \chi_j} - 1 ) ( e^{\chi_i - \chi_k} - 1 ) } 
        + \sum_{j \neq i} \frac{ \Delta_m(\chi) }{
        ( e^{\chi_i - \chi_j} - 1 ) 
        ( e^{\chi_j - \chi_i} - 1 ) } 
        \,. 
\end{aligned}
\end{equation}

Observe that the diagonal terms with $j=k$ no longer cancel because of the extra $e^{\chi_i-\chi_j}$ term, and the identity 
$\frac{1}{(a-b)(a-c)} 
    + \frac{1}{(b-a)(b-c)} 
    + \frac{1}{(c-a)(c-b)} 
    = 0$
cannot be used here.

The mixed triple sum is handled by the following symmetry identity. 

\begin{lemma}[Three Index Symmetry Sum]
\label{lm:three_index_symmetry}
With \(Q_{ij}=e^{\chi_i-\chi_j}\), one has
\begin{equation}
\label{eq:three-index-symmetry-sum}
    \sum_{\substack{i,j,k\in[m]\\ i,j,k\ \mathrm{distinct}}}
    \frac{1}{(Q_{ij}-1)(Q_{ik}-1)} 
    = \frac{m(m-1)(m-2)}{3} \,. 
\end{equation}
\end{lemma}

\begin{proof}

Fix three distinct indices $i,j,k$, and set $a:=Q_{ij}$ and $b:=Q_{ik}$, so $Q_{jk}=Q_{ik}/Q_{ij}=b/a$.
Then a direct calculation gives
\begin{equation}
    \frac{1}{(a-1)(b-1)}
    + \frac{1}{(a^{-1}-1)(ba^{-1}-1)}
    + \frac{1}{(b^{-1}-1)(ab^{-1}-1)}
    = 1 \,. 
\end{equation}
Equivalently,
\begin{equation}
    \frac{1}{(Q_{ij}-1)(Q_{ik}-1)}
    + \frac{1}{(Q_{ji}-1)(Q_{jk}-1)}
    + \frac{1}{(Q_{ki}-1)(Q_{kj}-1)}
    = 1 \,. 
\end{equation}
For each unordered triple $\{i,j,k\}$, the three terms above account for one cyclic orientation, and the opposite cyclic orientation contributes another copy of $1$. Hence each unordered triple contributes $2$, so
\begin{equation}
    \sum_{\substack{i,j,k\in[m]\\ i,j,k\ \mathrm{distinct}}}
    \frac{1}{(Q_{ij}-1)(Q_{ik}-1)}
    = 2 \binom{m}{3}
    = \frac{m(m-1)(m-2)}{3} \,. 
\end{equation}

\end{proof}

In fact this leads to the following result as well. 

\begin{lemma}[$\Delta_m(\chi)$ is Almost Harmonic]
\label{lm:vandermonde_almost_harmonic}
\begin{equation}
\label{eq:vandermonde-almost-harmonic}
    \Delta \Delta_m(\chi) = \frac{m(m-1)(2m-1)}{6} \Delta_m(\chi) \,. 
\end{equation}
\end{lemma}

\begin{proof}

Write $a_{ij}:=(Q_{ij}-1)^{-1}$, so $\di_i\Delta_m/\Delta_m=\sum_{j\neq i}a_{ij}$.
Then
\begin{equation}
\begin{aligned}
    \frac{\di_{ii} \Delta_m}{\Delta_m}
    &=
        \di_i \left(
            \frac{\di_i \Delta_m}{\Delta_m}
        \right)
        + \left(
            \frac{\di_i \Delta_m}{\Delta_m}
        \right)^2
        \\
    &=
        \sum_{j\neq i}
        \left(
            - \frac{Q_{ij}}{(Q_{ij}-1)^2}
        \right)
        + \sum_{\substack{j\neq i \\ k\neq i}} a_{ij} a_{ik}
        \\
    &=
        \sum_{\substack{j\neq i \\ k\neq i,\,k\neq j}}
        \frac{1}{(Q_{ij}-1)(Q_{ik}-1)}
        - \sum_{j\neq i} \frac{1}{Q_{ij}-1} \,. 
\end{aligned}
\end{equation}
Summing over $i$ gives
\begin{equation}
    \frac{\Delta \Delta_m}{\Delta_m}
    = \sum_{\substack{i,j,k\in[m]\\ i,j,k\ \mathrm{distinct}}}
        \frac{1}{(Q_{ij}-1)(Q_{ik}-1)}
        - \sum_{i\neq j} \frac{1}{Q_{ij}-1} \,. 
\end{equation}
By \cref{lm:three_index_symmetry}, the first sum is $m(m-1)(m-2)/3$.
For the second sum, pairing $(i,j)$ with $(j,i)$ gives $-(Q_{ij}-1)^{-1}-(Q_{ji}-1)^{-1}=1$, and hence $-\sum_{i\neq j}(Q_{ij}-1)^{-1}=\binom{m}{2}=m(m-1)/2$.

Combining the two contributions,
\begin{equation}
    \frac{\Delta \Delta_m}{\Delta_m}
    = \frac{m(m-1)(m-2)}{3} + \frac{m(m-1)}{2}
    = \frac{m(m-1)(2m-1)}{6} \,, 
\end{equation}
which proves the claim. 

\end{proof}

The following standard type-$A$ conditioning argument specializes \citet[Section~5.2]{bianeBougerolOConnell2005littelmann}; see also \citet[Equation~(5) and Appendix~A]{takahashiKatori2012noncolliding} for the explicit determinant form.

\begin{proposition}[Conditioning Representation]
\label{prop:unit-noise-conditioning}
Let $\mathcal{W}_m:=\{\chi\in\mathbb{R}^m:\chi_1<\cdots<\chi_m\}$ and $\mu_i:=i-m-1$. For $\chi\in\mathcal{W}_m$, let $B_1,\ldots,B_m$ be independent standard Brownian motions and set
\begin{equation}
\label{eq:drifted-brownian-conditioning-process}
    X_i(t):=\chi_i+B_i(t)+\mu_i t \,,\quad i\in[m] \,;
    \qquad
    \tau:=\inf\left\{t\geq 0:X_i(t)=X_{i+1}(t)
    \text{ for some }i\in[m-1]\right\} \,.
\end{equation}
Conditioned on $\{\tau=\infty\}$, the process $X$ has generator $\frac12\Delta+\nabla\log h\cdot\nabla$, where $h(\chi):=\prod_{i<j}(e^{-\chi_i}-e^{-\chi_j})\exp(-\sum_{i=1}^m\chi_i)$. Hence it is the diffusion obtained from \cref{eq:log_geom_dyson} by replacing the Brownian coefficient $\sqrt{2}$ by $1$.
\end{proposition}

\begin{proof}
Since $\mu_j=j-m-1$, the Vandermonde determinant gives the alternant representation
\begin{equation}
\label{eq:conditioning-alternant-survival}
\begin{aligned}
    h(\chi)
    &= \exp\left(-\sum_{i=1}^m\chi_i\right)\Delta_m(\chi)
    = \det\left[e^{\mu_j\chi_i}\right]_{i,j=1}^m \,,
    \\
    H_\mu(\chi)
    &:=e^{-\langle\mu,\chi\rangle}h(\chi)
    =\prod_{i<j}\left(1-e^{-(\chi_j-\chi_i)}\right) \,.
\end{aligned}
\end{equation}
Thus $0<H_\mu<1$ on $\mathcal{W}_m$, and $H_\mu$ extends continuously as zero to its collision boundary.

Every term in the determinant in \cref{eq:conditioning-alternant-survival} is an exponential whose exponent vector is a permutation of $\mu$. Consequently, with $\mathcal{L}_\mu:=\frac12\Delta+\mu\cdot\nabla$,
\begin{equation}
\label{eq:conditioning-survival-harmonicity}
\begin{aligned}
    \Delta h
    &= \|\mu\|^2 h \,,
    &
    \|\mu\|^2
    &= \sum_{k=1}^m k^2
    = \frac{m(m+1)(2m+1)}{6} \,,
    \\
    \mathcal{L}_\mu H_\mu
    &= e^{-\langle\mu,\chi\rangle}
    \left(\frac12\Delta h-\frac12\|\mu\|^2h\right)
    =0 \,.
\end{aligned}
\end{equation}
It follows from It\^{o}'s formula that $H_\mu(X(t\wedge\tau))$ is a bounded martingale. Since $H_\mu$ vanishes on the collision boundary,
\begin{equation}
\label{eq:conditioning-survival-martingale}
    H_\mu(\chi)
    = \mathbb{E}_\chi^\mu\left[H_\mu(X(t\wedge\tau))\right]
    = \mathbb{E}_\chi^\mu\left[
        \mathds{1}_{\{\tau>t\}}H_\mu(X(t))
      \right] \,.
\end{equation}
On $\{\tau=\infty\}$, the strong law for Brownian motion gives $(X_j(t)-X_i(t))/t\to\mu_j-\mu_i=j-i>0$, so \cref{eq:conditioning-alternant-survival} implies $H_\mu(X(t))\to 1$. On $\{\tau<\infty\}$, the indicator in \cref{eq:conditioning-survival-martingale} eventually vanishes. Dominated convergence therefore yields $H_\mu(\chi)=\mathbb{P}_\chi^\mu(\tau=\infty)$, which is positive by the product formula in \cref{eq:conditioning-alternant-survival}.

Let $P_t^{\mathcal{W}_m}$ denote the semigroup of centered Brownian motion killed on leaving $\mathcal{W}_m$. The Markov property and the Cameron--Martin formula give the conditional semigroup and its generator as
\begin{equation}
\label{eq:conditioning-doob-semigroup-generator}
\begin{aligned}
    Q_t f(\chi)
    &:= \mathbb{E}_\chi^\mu\left[f(X(t))\mid\tau=\infty\right]
    = \frac{1}{H_\mu(\chi)}
      \mathbb{E}_\chi^\mu\left[
        \mathds{1}_{\{\tau>t\}}H_\mu(X(t))f(X(t))
      \right]
    \\
    &= e^{-\|\mu\|^2t/2}
       \frac{P_t^{\mathcal{W}_m}(hf)(\chi)}{h(\chi)} \,,
    \\
    \mathcal{L}_\mu^{H_\mu}f
    &:= H_\mu^{-1}\mathcal{L}_\mu(H_\mu f)
    = \frac12\Delta f+\left(\mu+\nabla\log H_\mu\right)\cdot\nabla f
    = \frac12\Delta f+\nabla\log h\cdot\nabla f \,.
\end{aligned}
\end{equation}
The logarithmic derivative in \cref{eq:modified-vandermonde-log-derivative} now identifies the resulting diffusion with the coefficient-one version of \cref{eq:log_geom_dyson}. The factor $e^{-\|\mu\|^2t/2}$ in \cref{eq:conditioning-doob-semigroup-generator} is the scalar eigenvalue normalization; the underlying killed process is absorbed only at the first collision.
\end{proof}

\section{Collision Avoidance and Instantaneous Simplicity}
\label{sec:appendix-collision}

\begin{lemma}[Matrix Realization of the Covariance SDE]
\label{lem:matrix-realization-covariance-sde}
Let $W_t$ be an $m\times m$ Brownian matrix, and let $H_t$ solve
\begin{equation}
    dH_t
    = \frac{1}{\sqrt{2}} \, dW_t \, H_t
    - \frac{m}{4} H_t \, dt \,,
    \qquad
    H_0 = \Phi_0^{1/2} \,.
\end{equation}
Set $\Phi_t := H_t^\top H_t$. Then $\Phi_t$ solves the covariance SDE of \cref{thm:lin_cov_sde}. Moreover, if $\Phi_0 \in \SPD(m)$, then $\Phi_t \in \SPD(m)$ for all $t \geq 0$ almost surely.
\end{lemma}

\begin{proof}

Apply It\^o's product rule to $\Phi_t = H_t^\top H_t$:
\begin{equation}
    d\Phi_t = dH_t^\top H_t + H_t^\top dH_t + dH_t^\top dH_t \,.
\end{equation}
Substituting the defining SDE for $H_t$ gives
\begin{equation}
\begin{aligned}
    d\Phi_t
    &= \frac{1}{\sqrt{2}} H_t^\top dW_t^\top H_t
    + \frac{1}{\sqrt{2}} H_t^\top dW_t H_t
    - \frac{m}{2} \Phi_t \, dt
    + \frac{1}{2} H_t^\top dW_t^\top dW_t H_t \,.
\end{aligned}
\end{equation}
Since $dW_t^\top dW_t = m I \, dt$, the drift cancels and
\begin{equation}
    d\Phi_t
    = \frac{1}{\sqrt{2}} H_t^\top dW_t^\top H_t
    + \frac{1}{\sqrt{2}} H_t^\top dW_t H_t \,.
\end{equation}
Writing $H_t = O_t \Phi_t^{1/2}$ with $O_t \in O(m)$ and using the orthogonal invariance of the Brownian matrix, this has the same law as
\begin{equation}
    d\Phi_t
    = \frac{1}{\sqrt{2}} \, \Phi_t^{1/2} d\widetilde W_t \Phi_t^{1/2}
    + \frac{1}{\sqrt{2}} \, \Phi_t^{1/2} d\widetilde W_t^\top \Phi_t^{1/2} \,,
\end{equation}
whose covariance is precisely
\begin{equation}
    \mathbb{E}\bigl[
        d\Phi_{\alpha\beta} d\Phi_{\gamma\delta}
        \,\big|\,
        \mathcal{F}_t
    \bigr]
    =
    \bigl(
        \Phi_{\alpha\gamma}\Phi_{\beta\delta}
        + \Phi_{\alpha\delta}\Phi_{\beta\gamma}
    \bigr)
    \, dt \,.
\end{equation}
Equivalently, with
$S_t=(\widetilde W_t+\widetilde W_t^\top)/\sqrt{2}$, the preceding display is the matrix SDE \cref{eq:cov_sde_matrix}. The displayed covariance is exactly the upper-triangular coordinate form \cref{eq:cov_sde}.

It remains to prove positive definiteness. Since $\Phi_t = H_t^\top H_t$, it suffices to show that $H_t$ remains invertible. Up to the first time at which $\det H_t$ hits $0$, right multiplication by $H_t^{-1}$ gives
\begin{equation}
    dH_t H_t^{-1}
    = \frac{1}{\sqrt{2}} \, dW_t - \frac{m}{4} I \, dt \,,
\end{equation}
Using It\^o's formula for $\log \det H_t$, equivalently in right-logarithmic form, yields
\begin{equation}
\begin{aligned}
    d \log \det H_t
    &= \operatorname{Tr}( dH_t H_t^{-1} )
    - \frac{1}{2} \operatorname{Tr}( dH_t H_t^{-1} dH_t H_t^{-1} ) \\
    &= \frac{1}{\sqrt{2}} \operatorname{Tr}(dW_t)
    - \frac{m^2}{4} dt
    - \frac{1}{4} \operatorname{Tr}(dW_t dW_t) \\
    &= \frac{1}{\sqrt{2}} \operatorname{Tr}(dW_t)
    - \frac{m(m+1)}{4} dt \,.
\end{aligned}
\end{equation}
This is a real-valued continuous semimartingale with finite coefficients on every finite time interval, so $\det H_t$ cannot hit $0$ in finite time once $\det H_0 > 0$. Hence $H_t$ remains invertible and therefore $\Phi_t \in \SPD(m)$ for all $t \geq 0$ almost surely.

\end{proof}

\begin{lemma}[Instantaneous Simplicity and Absence of Collisions]
\label{lem:instantaneous-simplicity-no-collision}
Let $\Phi_t$ solve the covariance SDE of \cref{thm:lin_cov_sde} with $\Phi_0 \in \SPD(m)$, and let $\lambda_1(t) \leq \cdots \leq \lambda_m(t)$ be the ordered eigenvalues of $\Phi_t$.
If $\Phi_0$ has simple spectrum, then the eigenvalues satisfy \eqref{eq:gdbm_sde} for all $t\geq 0$, and their collision time is almost surely infinite.
For arbitrary $\Phi_0\in\SPD(m)$ and every fixed $\varepsilon>0$, the spectrum of $\Phi_\varepsilon$ is simple almost surely; consequently, on every interval $[\varepsilon,\infty)$ the eigenvalues satisfy \eqref{eq:gdbm_sde} and never collide.
\end{lemma}

\begin{proof}

First suppose that $\Phi_0$ has simple spectrum. By \cref{lem:matrix-realization-covariance-sde}, the covariance SDE can be written in the matrix form
\begin{equation}
    d\Phi_t
    = g(\Phi_t) \, dW_t \, h(\Phi_t)
    + h(\Phi_t) \, dW_t^\top \, g(\Phi_t) \,,
\end{equation}
Here $g(x)=h(x)=2^{-1/4}\sqrt{x}$ and $b(x)\equiv 0$.
For these coefficients, $2g(x)h(x)=\sqrt{2}\,x$ and $G(x,y):=g(x)^2h(y)^2+g(y)^2h(x)^2=xy$.
Moreover, $g(x)^2=h(x)^2=x/\sqrt{2}$ is Lipschitz on $\mathbb{R}_{\geq 0}$, and $g(x)^2h(x)^2=x^2/2$ is convex.
Therefore Theorem 3 and Theorem 5 of \citet{graczyk2013multidimensional} apply: up to the first collision time the eigenvalues satisfy \eqref{eq:gdbm_sde}, and the collision time is almost surely infinite. This proves the simple-spectrum case.

For the general initial condition, fix $\varepsilon > 0$. By \cref{lem:matrix-realization-covariance-sde}, the process stays in $\SPD(m)$ for all times. On $\SPD(m)$, the covariance matrix $\Sigma(\Phi)$ is strictly positive definite. Indeed, for every symmetric matrix $A$,
\begin{equation}
    \sum_{\alpha,\beta,\gamma,\delta}
    A_{\alpha\beta} \, \Sigma(\Phi)_{\alpha\beta,\gamma\delta} \, A_{\gamma\delta}
    = 2 \, \Tr(A\Phi A\Phi)
    = 2 \, \|\Phi^{1/2} A \Phi^{1/2}\|_F^2 > 0
\end{equation}
whenever $A \neq 0$. Hence the vectorized diffusion is locally elliptic on $\SPD(m)$, and standard elliptic diffusion theory yields that $\Phi_\varepsilon$ has a density with respect to Lebesgue measure on $\Sym(m)$ for every fixed $\varepsilon > 0$; see, for example, \citet{Bally2006}.

The subset of symmetric matrices with repeated eigenvalues is the zero set of the discriminant polynomial of the characteristic polynomial, hence is Lebesgue-null. Therefore $\Phi_\varepsilon$ has simple spectrum almost surely.

Now condition on $\mathcal{F}_\varepsilon$. On the event that $\Phi_\varepsilon$ has simple spectrum, the simple-spectrum case applies to the time-shifted process by the strong Markov property, so no collision can occur after time $\varepsilon$. Since this holds for every rational $\varepsilon > 0$, collisions cannot occur at any positive time. The SDE therefore holds on every interval $[\varepsilon,\infty)$.

\end{proof}

\section{Information-Geometric Interpretation}
\label{sec:appendix-dual-brownian}

This appendix records the coordinate calculation underlying \cref{rm:dual_brownian}.
While $\Phi_t$ is not the standard Brownian motion on the manifold $(\SPD(m),\Sigma^{-1})$, it can be viewed as a Brownian motion with respect to a different affine connection.

In information geometry \citep{amari2016information}, we can consider the geometric structure arising from a family of distributions and a divergence.
In the present case, the relevant family is the zero-mean Gaussian distributions on $\mathbb{R}^m$ with the KL-divergence, which naturally induces the affine-invariant metric \cref{prop:sigma_affine_invar}.
To avoid notational ambiguity, write $\nabla^{\mathrm{LC}}$ for the Levi--Civita connection, $\nabla^{(+1)}$ for the primal flat connection, and $\nabla^{(-1)}$ for the dual connection.
Then the $\alpha$-connections are given by
\begin{equation}
    \nabla^{(\alpha)}
    := \frac{1+\alpha}{2}\,\nabla^{(+1)} + \frac{1-\alpha}{2}\,\nabla^{(-1)} \,,
    \qquad
    \nabla^{(0)} = \nabla^{\mathrm{LC}} \,.
\end{equation}

Let $x = (x^1,\ldots,x^N)$, with $N = m(m+1)/2$, denote affine coordinates on $\SPD(m)$ induced by the upper triangular entries.
If
\begin{equation}
    \omega^{(\alpha)} = p^{(\alpha)}(x) \, dx^1 \cdots dx^N
\end{equation}
is the $\alpha$-volume form, then using $\omega^{(\alpha)} = f^{(\alpha)} dv$ together with
\begin{equation}
    f^{(\alpha)} = C (\det g)^{-\alpha/2} \,,
    \qquad
    dv = \sqrt{\det g} \, dx^1 \cdots dx^N \,,
\end{equation}
we obtain
\begin{equation}
    p^{(\alpha)}(x) = C (\det g)^{(1-\alpha)/2} \,.
\end{equation}
Hence the divergence of a vector field $X = X^i \di_i$ with respect to $\omega^{(\alpha)}$ is
\begin{equation}
    \operatorname{div}^{(\alpha)} X
    = \frac{1}{p^{(\alpha)}} \di_i \bigl( p^{(\alpha)} X^i \bigr)
    = \di_i X^i + \frac{1-\alpha}{2} X(\log \det g) \,.
\end{equation}

For a smooth function $\varphi$, define the $\alpha$-Laplacian by
\begin{equation}
    \Delta^{(\alpha)} \varphi := \operatorname{div}^{(\alpha)}(\grad \varphi) \,,
    \qquad
    (\grad \varphi)^i = g^{ij} \di_j \varphi \,.
\end{equation}
Therefore
\begin{equation}
    \Delta^{(\alpha)} \varphi
    = \di_i \bigl( g^{ij} \di_j \varphi \bigr)
    + \frac{1-\alpha}{2} g^{ij} \bigl( \di_i \log \det g \bigr) \di_j \varphi \,.
\end{equation}
Because the affine-invariant metric is Hessian in these affine coordinates--more precisely, for the potential $\psi(\Phi):=-\frac{1}{2}\log\det\Phi$, one has $g_{ij}=\partial_i\partial_j\psi$--the cofactor identity
\begin{equation}
    \di_i \bigl( (\det g) g^{ij} \bigr) = 0 \,.
\end{equation}
follows from equality of mixed partial derivatives.

Thus the two cases relevant here simplify to
\begin{equation}
\begin{aligned}
    \Delta^{(-1)} \varphi
    &= g^{ij} \di_i \di_j \varphi \,, \\
    \Delta^{(+1)} \varphi
    &= g^{ij} \di_i \di_j \varphi
        - \langle \grad \varphi, \grad \log \det g \rangle \,.
\end{aligned}
\end{equation}
Let $L$ denote the It\^o generator of the covariance diffusion in \cref{thm:lin_cov_sde}.
In the upper-triangular coordinate form, using \cref{prop:sigma_affine_invar}, we can write
\begin{equation}
\begin{aligned}
    L \varphi
    &= \frac{1}{2} \sum_{a,b} \Sigma(\Phi)_{ab} \, \di_{ab} \varphi
    &= \frac{1}{2} g^{ij} \di_i\di_j \varphi
    = \frac{1}{2} \Delta^{(-1)} \varphi \,.
\end{aligned}
\end{equation}
Thus the It\^o generator of $\Phi_t$ is $L = \frac{1}{2} \Delta^{(-1)}$, which corresponds to the dual connection Laplacian.
Accordingly, $\Phi_t$ is Brownian motion on $\SPD(m)$ associated with the dual connection $\nabla^{(-1)}$ under the standard probabilistic convention that Brownian motion has generator $\frac{1}{2}\Delta^{(-1)}$.

\section{Mean-Field Limit for the Geometric Dyson Brownian Motion}
\label{sec:appendix-meanfield-gdbm}

This appendix collects the compactness, a priori estimates, and real-axis uniqueness arguments used in the proof of \cref{thm:burgers-fixed-point}. The formal characteristic and free-probability derivations are given in the proof of \cref{cor:fixed-point-free-flow}; the role of this appendix is to justify the limiting weak equation, select the unique physical branch, and record the inverse-branch identity used in the $S$-transform comparison. Throughout this appendix, we work on the fixed terminal spectral-time interval $[0,\bar\tau]$, assume $\rho_0^{(m)}\xrightarrow{d}\rho_0$ and $\sup_m\langle x,\rho_0^{(m)}\rangle<\infty$, and let
\begin{equation}
    \rho_\tau^{(m)}
    :=
    \frac{1}{m}\sum_{i=1}^m\delta_{\mu_{i,\tau}},
    \qquad
    0\le \tau\le \bar\tau,
\end{equation}
where $(\mu_{i,\tau})_{i=1}^m$ is the time-changed geometric Dyson Brownian motion
\begin{equation}
    d\mu_{i,\tau}
    =
    \sqrt{\frac{2}{m}}\,\mu_{i,\tau}\,dB_{i,\tau}
    +
    \frac{1}{m}\sum_{j\ne i}
        \frac{\mu_{i,\tau}\mu_{j,\tau}}
             {\mu_{i,\tau}-\mu_{j,\tau}}\,d\tau \, .
\end{equation}

When the initial condition has repeated eigenvalues, the identities below are first justified on each interval $[\varepsilon,\bar\tau]$ by \cref{lem:instantaneous-simplicity-no-collision}, and then one lets $\varepsilon\downarrow 0$. We suppress this approximation in the notation.

We use the test class $\mathcal F_b$ and the kernel $H_f$ from \eqref{eq:meanfield-test-class} and \eqref{eq:meanfield-Hf}. For the kernel estimates below, it is convenient also to use the corresponding local class
\begin{equation}
\label{eq:meanfield-Floc-appendix}
    \mathcal F_{\mathrm{loc}}
    :=
    \left\{
        f\in C^2(\mathbb{R}_{\geq 0}) :
        \|xf'(x)\|_\infty+\|x^2f''(x)\|_\infty<\infty
    \right\},
\end{equation}
with $H_f$ defined by the same formula as in \eqref{eq:meanfield-Hf}.

\begin{lemma}[Compactness and Identification of Subsequential Limits]
\label{lem:meanfield-technical}
The sequence $(\rho_\tau^{(m)})_{\tau\in[0,\bar\tau]}$ is tight in $C([0,\bar\tau],\mathcal{P}(\mathbb{R}_{\geq 0}))$. Every subsequential limit $(\rho_\tau)_{\tau\in[0,\bar\tau]}$ is a continuous path and, for every fixed $f\in\mathcal F_b$, satisfies the weak equation \eqref{eq:limit-weak} almost surely for all $\tau\in[0,\bar\tau]$.
\end{lemma}

\begin{proof}
We divide the proof into four steps.

\smallskip
\noindent
\emph{Step 1: $H_f$ is bounded and continuous.}
Fix $f\in\mathcal F_{\mathrm{loc}}$. Continuity away from the diagonal is obvious. On the diagonal, \eqref{eq:meanfield-Hf} gives the continuous extension $H_f(x,x)=x^2 f''(x)$. It remains to show boundedness. Let
\begin{equation}
    A_f := \|x f'(x)\|_\infty \,,
    \qquad
    B_f := \|x^2 f''(x)\|_\infty \,. 
\end{equation}
If $x$ and $y$ are comparable, say $x/2\le y\le 2x$, then by the mean-value theorem,
\begin{equation}
    H_f(x,y)=xy f''(\xi)
\end{equation}
for some $\xi$ between $x$ and $y$, and therefore
\begin{equation}
    |H_f(x,y)|
    \le
    xy \frac{B_f}{\xi^2}
    \le 4 B_f \,. 
\end{equation}
If $x>2y$, then $|x-y|\ge x/2$, hence
\begin{equation}
    |H_f(x,y)|
    \le
    \frac{2xy}{x}\bigl(|f'(x)|+|f'(y)|\bigr)
    \le
    2y|f'(x)|+2y|f'(y)|
    \le
    4A_f \,. 
\end{equation}
The case $y>2x$ is symmetric. Thus
\begin{equation}
    \|H_f\|_\infty \le 4A_f+4B_f \,. 
\end{equation}

\smallskip
\noindent
\emph{Step 2: finite-$m$ semimartingale identity.}
Applying It\^o's formula to $\langle f,\rho_\tau^{(m)}\rangle=m^{-1}\sum_{i=1}^m f(\mu_{i,\tau})$ gives
\begin{equation}
\begin{aligned}
    d\langle f,\rho_\tau^{(m)}\rangle
    &=
    \frac{\sqrt{2}}{m^{3/2}}
    \sum_{i=1}^m
    \mu_{i,\tau}f'(\mu_{i,\tau})\,dB_{i,\tau}
    \\
    &\quad
    +
    \frac{1}{m^2}
    \sum_{i=1}^m
    f'(\mu_{i,\tau})
    \sum_{j\ne i}
    \frac{\mu_{i,\tau}\mu_{j,\tau}}
         {\mu_{i,\tau}-\mu_{j,\tau}}\,d\tau
    \\
    &\quad
    +
    \frac{1}{m^2}
    \sum_{i=1}^m
    \mu_{i,\tau}^2 f''(\mu_{i,\tau})\,d\tau \,.
\end{aligned}
\end{equation}
Define
\begin{equation}
    M_\tau^{(m,f)}
    :=
    \frac{\sqrt{2}}{m^{3/2}}
    \sum_{i=1}^m
    \int_0^\tau \mu_{i,s}f'(\mu_{i,s})\,dB_{i,s} \,.
\end{equation}
Its quadratic variation is
\begin{equation}
    \langle M^{(m,f)}\rangle_\tau
    =
    \frac{2}{m^2}
    \int_0^\tau
    \langle x^2 (f'(x))^2,\rho_s^{(m)}\rangle\,ds \,.
    \label{eq:finite-m-bracket}
\end{equation}

For the drift term, symmetrization yields
\begin{equation}
\begin{aligned}
    \sum_{i=1}^m
    f'(\mu_{i,\tau})
    \sum_{j\ne i}
    \frac{\mu_{i,\tau}\mu_{j,\tau}}
         {\mu_{i,\tau}-\mu_{j,\tau}}
    &=
    \frac{1}{2}\sum_{i\ne j}
    \frac{\mu_{i,\tau}\mu_{j,\tau}
    \bigl(f'(\mu_{i,\tau})-f'(\mu_{j,\tau})\bigr)}
    {\mu_{i,\tau}-\mu_{j,\tau}}
    \\
    &=
    \frac{1}{2}\sum_{i\ne j}
    H_f(\mu_{i,\tau},\mu_{j,\tau}) \,.
\end{aligned}
\end{equation}
Combining the It\^o formula with the symmetrized drift and integrating in time gives
\begin{equation}
\begin{aligned}
    \langle f,\rho_\tau^{(m)}\rangle
    &=
    \langle f,\rho_0^{(m)}\rangle
    +
    M_\tau^{(m,f)}
    +
    \frac{1}{2}
    \int_0^\tau
    \iint_{(\mathbb{R}_{\geq 0})^2}
    H_f(x,y)\,\rho_s^{(m)}(dx)\rho_s^{(m)}(dy)\,ds
    \\
    &\qquad
    +
    \frac{1}{2m}
    \int_0^\tau
    \langle x^2 f''(x),\rho_s^{(m)}\rangle\,ds \,.
\end{aligned}
\label{eq:finite-m-semi}
\end{equation}

\smallskip
\noindent
\emph{Step 3: tightness.}
First we establish compact containment. Let
\begin{equation}
    A_\tau^{(m)}
    :=
    \langle x,\rho_\tau^{(m)}\rangle
    =
    \frac{1}{m}\sum_{i=1}^m \mu_{i,\tau} \,.
\end{equation}
A direct computation from the particle system shows that the drift cancels pairwise:
\begin{equation}
    \sum_{i=1}^m\sum_{j\ne i}
    \frac{\mu_{i,\tau}\mu_{j,\tau}}{\mu_{i,\tau}-\mu_{j,\tau}}=0 \,.
\end{equation}
Hence $A^{(m)}$ is a nonnegative local martingale, and therefore a supermartingale. In particular, $\mathbb{E}[A_\tau^{(m)}]\le A_0^{(m)}=\langle x,\rho_0^{(m)}\rangle$. By the maximal inequality for nonnegative supermartingales,
\begin{equation}
    \mathbb{P}\!\left(\sup_{0\le s\le \bar\tau} A_s^{(m)} > R\right)\le \frac{A_0^{(m)}}{R} \,.
\end{equation}
The sets $K_R:=\{\mu\in\mathcal{P}(\mathbb{R}_{\geq 0}): \int x\,\mu(dx)\le R\}$ are compact in $\mathcal{P}(\mathbb{R}_{\geq 0})$: Markov's inequality gives tightness, and lower semicontinuity of the first moment gives closedness. Hence, by the standing uniform first-moment assumption, the bound above yields the compact-containment estimate needed for tightness.

Next fix $f\in C_c^\infty(\mathbb{R}_{\geq 0})$. Since $H_f$ is bounded, the drift term in \eqref{eq:finite-m-semi} is uniformly Lipschitz in time:
\begin{equation}
    \left|
        \frac{1}{2}\int_r^u \iint H_f\,d\rho_q^{(m)}d\rho_q^{(m)}\,dq
        +
        \frac{1}{2m}\int_r^u \langle x^2 f''(x),\rho_q^{(m)}\rangle\,dq
    \right|
    \le C_f |u-r| \,.
\end{equation}
The martingale term vanishes uniformly in $m$ because \eqref{eq:finite-m-bracket} implies
\begin{equation}
    \mathbb{E}\left[\langle M^{(m,f)}\rangle_{\bar\tau}\right]
    \le
    \frac{2\bar\tau}{m^2}\|x f'(x)\|_\infty^2 \,.
\end{equation}
By the Burkholder--Davis--Gundy inequality,
\begin{equation}
    \mathbb{E}\left[\sup_{0\le r\le \bar\tau}|M_r^{(m,f)}|\right]\longrightarrow 0 \,.
\end{equation}
Therefore the family $(\langle f,\rho_\tau^{(m)}\rangle)_{\tau\in[0,\bar\tau]}$ is tight in $C([0,\bar\tau],\mathbb{R})$ for every $f\in C_c^\infty(\mathbb{R}_{\geq 0})$.

At this point one may repeat the standard compactness argument used in \citet[Proposition 4.3.10 and Lemma 4.3.13]{anderson2010introduction}: compact containment plus tightness of a countable separating family of test-function evaluations implies tightness of $(\rho_\tau^{(m)})_{\tau\in[0,\bar\tau]}$ in $C([0,\bar\tau],\mathcal{P}(\mathbb{R}_{\geq 0}))$.

\smallskip
\noindent
\emph{Step 4: identification of subsequential limits.}
Let $(\rho^{(m_k)})_{k\ge 1}$ be a subsequence converging in law in $C([0,\bar\tau],\mathcal{P}(\mathbb{R}_{\geq 0}))$ to some limit $(\rho_\tau)_{\tau\in[0,\bar\tau]}$. Continuity of evaluation at time zero, together with $\rho_0^{(m_k)}\xrightarrow{d}\rho_0$, implies that the limit path starts from the prescribed measure $\rho_0$ almost surely. Fix $f\in\mathcal F_b$. For a path $\mu\in C([0,\bar\tau],\mathcal{P}(\mathbb{R}_{\geq 0}))$, define $\Psi_f(\mu)\in C([0,\bar\tau],\mathbb R)$ by
\begin{equation}
\begin{aligned}
    \Psi_f(\mu)_\tau
    &:=
    \langle f,\mu_\tau\rangle
    -\langle f,\mu_0\rangle
    -\frac{1}{2}\int_0^\tau
    \iint H_f(x,y)\,\mu_s(dx)\mu_s(dy)\,ds.
\end{aligned}
\end{equation}
Because $f$ and $H_f$ are bounded and continuous, $\Psi_f$ is continuous for the uniform weak topology on path space. The finite-$m$ identity \eqref{eq:finite-m-semi} gives
\begin{equation}
    \Psi_f(\rho^{(m_k)})_\tau
    =
    M_\tau^{(m_k,f)}
    +
    \frac{1}{2m_k}
    \int_0^\tau
    \langle x^2 f''(x),\rho_s^{(m_k)}\rangle\,ds.
\end{equation}

The correction term vanishes uniformly because
\begin{equation}
    \sup_{0\le \tau\le \bar\tau}
    \left|
        \frac{1}{2m_k}
        \int_0^\tau \langle x^2 f''(x),\rho_s^{(m_k)}\rangle\,ds
    \right|
    \le
    \frac{\bar\tau}{2m_k}\|x^2 f''(x)\|_\infty
    \longrightarrow 0 \,. 
\end{equation}
The martingale term also vanishes in probability uniformly on $[0,\bar\tau]$, since
\begin{equation}
    \mathbb{E}\left[\sup_{0\le r\le \bar\tau}|M_r^{(m_k,f)}|\right]\to 0 \,.
\end{equation}
Thus $\Psi_f(\rho^{(m_k)})\to0$ in probability in $C([0,\bar\tau],\mathbb R)$. On the other hand, the continuous mapping theorem gives $\Psi_f(\rho^{(m_k)})\xrightarrow{d}\Psi_f(\rho)$. Hence $\Psi_f(\rho)=0$ almost surely, which is \eqref{eq:limit-weak}. This completes the proof.
\end{proof}

\begin{lemma}[Real Characteristics, Inverse Branch, and Uniqueness]
\label{lem:meanfield-uniqueness}
Let $(\rho_\tau)_{\tau\in[0,\bar\tau]}$ be a continuous path in $\mathcal{P}(\mathbb{R}_{\geq 0})$ satisfying \eqref{eq:limit-weak} for every $f\in\mathcal F_b$. Define
\begin{equation}
    G_\tau(z):=\int_{\mathbb{R}_{\geq 0}}\frac{x}{z-x}\,\rho_\tau(dx),
    \qquad z\in D:=\mathbb{C}\setminus\mathbb{R}_{\geq 0} \,.
    \label{eq:appendix-t-transform}
\end{equation}
For each $0\le \tau\le \bar\tau$ and $u<0$,
\begin{equation}
    G_\tau(u)=G_0\!\left(ue^{-\tau G_\tau(u)}\right).
\end{equation}
Equivalently, set $H_\tau(v):=v e^{\tau G_0(v)}$ for $v<0$. Then $H_\tau$ is a strictly increasing bijection from $(-\infty,0)$ onto $(-\infty,0)$, and
\begin{equation}
    G_\tau(u)=G_0\!\left(H_\tau^{-1}(u)\right),
    \qquad u<0,\quad \tau\in[0,\bar\tau].
    \label{eq:real-inverse-formula}
\end{equation}
If, in addition, $\rho_0$ is compactly supported on $\mathbb{R}_{\geq 0}$ and has nonzero first moment, and if $K_0$ and $K_\tau$ denote the inverse branches of $G_0$ and $G_\tau$ corresponding to $z=\infty$, then, as meromorphic germs at $w=0$,
\begin{equation}
    K_\tau(w)=e^{\tau w}K_0(w).
    \label{eq:inverse-branch-flow-germ}
\end{equation}
Consequently, there is at most one continuous measure-valued path satisfying \eqref{eq:limit-weak} with initial law $\rho_0$.
\end{lemma}

\begin{proof}
Using $f_z(x)=x/(z-x)$, the transform calculation in the proof of \cref{thm:burgers-fixed-point} applies to \eqref{eq:limit-weak} and gives
\begin{equation}
    \partial_\tau G_\tau(z)=-zG_\tau(z)\partial_zG_\tau(z),
    \qquad z\in D \,.
    \label{eq:appendix-burgers-equation}
\end{equation}
For $u<0$,
\begin{equation}
    -1\le G_s(u)\le 0,
    \qquad
    \partial_uG_s(u)
    =
    -\int_{\mathbb{R}_{\geq 0}}\frac{x}{(u-x)^2}\,\rho_s(dx)\le 0 \,.
    \label{eq:real-axis-bounds}
\end{equation}
Moreover, on compact subintervals of $(-\infty,0)$, the derivative $\partial_uG_s(u)$ is locally bounded uniformly in $s\in[0,\bar\tau]$.

Fix $\tau\in[0,\bar\tau]$ and $u<0$. Let $X_s$, $0\le s\le\tau$, solve
\begin{equation}
    \dot X_s=X_sG_s(X_s),
    \qquad X_\tau=u \,.
    \label{eq:real-characteristic-ode}
\end{equation}
The bounds in \eqref{eq:real-axis-bounds} imply that the characteristic remains in $(-\infty,0)$. Along this curve,
\begin{equation}
\begin{aligned}
    \frac{d}{ds}G_s(X_s)
    &=
    \partial_sG_s(X_s)+\dot X_s\,\partial_xG_s(X_s) \\
    &=
    -X_sG_s(X_s)\partial_xG_s(X_s)
    +X_sG_s(X_s)\partial_xG_s(X_s)
    =0 \,.
\end{aligned}
\label{eq:characteristic-cancellation}
\end{equation}
Thus $G_s(X_s)$ is constant, so $G_\tau(u)=G_0(X_0)$, while the characteristic equation gives $u=X_0e^{\tau G_\tau(u)}$. Therefore
\begin{equation}
    G_\tau(u)=G_0\!\left(u e^{-\tau G_\tau(u)}\right),
    \qquad u<0 \,.
    \label{eq:real-characteristic-formula}
\end{equation}

Now set $H_\tau(v):=v e^{\tau G_0(v)}$ for $v<0$. Since
\begin{equation}
    G_0'(v)
    =
    -\int_{\mathbb{R}_{\geq 0}}\frac{x}{(v-x)^2}\,\rho_0(dx)\le 0 \, ,
    \label{eq:G0-prime-negative}
\end{equation}
we obtain
\begin{equation}
    H_\tau'(v)
    =
    e^{\tau G_0(v)}\left(1+\tau v G_0'(v)\right)>0 \, ,
    \label{eq:Htau-increasing}
\end{equation}
because $v<0$ and $G_0'(v)\le0$. Also, $H_\tau(v)\to-\infty$ as $v\to-\infty$, while $H_\tau(v)\to0$ as $v\uparrow0$. Thus $H_\tau$ is a strictly increasing bijection from $(-\infty,0)$ onto $(-\infty,0)$. Combining this bijectivity with \eqref{eq:real-characteristic-formula} gives \eqref{eq:real-inverse-formula}.

It remains to record the corresponding inverse-branch identity. Assume now that $\rho_0$ is compactly supported and has nonzero first moment. For $v<0$ with $|v|$ sufficiently large, set $w=G_0(v)$. Then $w<0$ is sufficiently close to $0$, and $v=K_0(w)$ on the inverse branch corresponding to $z=\infty$. By \eqref{eq:real-inverse-formula}, with
\begin{equation}
    u=H_\tau(v)=v e^{\tau G_0(v)}=e^{\tau w}K_0(w),
\end{equation}
we have $G_\tau(u)=G_0(v)=w$. Since $u\to-\infty$ as $w\uparrow0$, this is the same large-$z$ inverse branch of $G_\tau$. Hence
\begin{equation}
    K_\tau(w)=e^{\tau w}K_0(w)
\end{equation}
for real $w<0$ sufficiently close to $0$. The two sides are meromorphic germs at $w=0$, so the identity theorem gives \eqref{eq:inverse-branch-flow-germ}.

Now suppose $(\rho_\tau)_{\tau\in[0,\bar\tau]}$ and $(\bar\rho_\tau)_{\tau\in[0,\bar\tau]}$ are two continuous paths satisfying \eqref{eq:limit-weak} with the same initial law $\rho_0$, and let $G_\tau$ and $\bar G_\tau$ be their $T$-transforms. Formula \eqref{eq:real-inverse-formula} gives $G_\tau(u)=\bar G_\tau(u)$ for all $u<0$. For each fixed $\tau$, both transforms are holomorphic on $D$, so the identity theorem gives $G_\tau(z)=\bar G_\tau(z)$ on $D$. The $T$-transform determines the underlying probability measure, since the Cauchy--Stieltjes transform is $(G_\tau(z)+1)/z$. Hence $\rho_\tau=\bar\rho_\tau$ for every $\tau\in[0,\bar\tau]$.
\end{proof}

\section{Support Analysis for the Free Log-Normal}

\begin{lemma}[Support from the Inverse Branch]
\label{lem:support-branch-analysis}
Fix $\tau>0$ and let $\nu_\tau$ be the free log-normal law from \cref{cor:fixed-point-free-flow}. Define $K_\tau(w):=e^{\tau w}\left(1+\frac{1}{w}\right)$ and $w_\pm:=(-1\pm\sqrt{1+4/\tau})/2$. Then
\begin{equation}
    \operatorname{supp}\nu_\tau=[K_\tau(w_-),K_\tau(w_+)].
\end{equation}
\end{lemma}

\begin{proof}
Let $G_\tau(z)=\int_{\mathbb{R}_{\geq 0}} x(z-x)^{-1}\nu_\tau(dx)$ be the $T$-transform of $\nu_\tau$. By \cref{eq:fixed_point}, $G_\tau$ is the physical branch of the fixed-point equation, so $K_\tau(G_\tau(z))=z$ on $\mathbb C\setminus\mathbb{R}_{\geq 0}$ and $G_\tau(z)\sim1/z$ as $z\to\infty$. Set $a_\tau:=K_\tau(w_-)$ and $b_\tau:=K_\tau(w_+)$.
The derivative is
\begin{equation}
    K_\tau'(w)=e^{\tau w}\frac{\tau w^2+\tau w-1}{w^2}.
\end{equation}
The two critical points are $w_-<-1$ and $w_+>0$. Moreover, using the critical-point identity $\tau w(w+1)=1$, one has $K_\tau(w)=e^{\tau w}/(\tau w^2)$ at critical points, and therefore $a_\tau<b_\tau$.

We first record the real range of $K_\tau$. On $(-\infty,-1]$, the function $K_\tau$ is positive, tends to $0$ at both endpoints $w=-\infty$ and $w=-1$, and has its unique maximum at $w_-$. Hence
\begin{equation}
    K_\tau((-\infty,-1])=[0,a_\tau].
\end{equation}
On $(-1,0)$, the function is negative and maps this interval onto $(-\infty,0)$. On $(0,\infty)$, the function is positive, tends to $+\infty$ as $w\downarrow0$ and as $w\to\infty$, and has its unique minimum at $w_+$. Hence
\begin{equation}
    K_\tau((0,\infty))=[b_\tau,\infty).
\end{equation}
Consequently,
\begin{equation}
    K_\tau(\mathbb R\setminus\{0\})=(-\infty,a_\tau]\cup[b_\tau,\infty).
\end{equation}

We now use the physical inverse branch. Since $G_\tau(z)\sim 1/z$ as $z\to\infty$, for large positive real $z$ the boundary value of $G_\tau(z)$ lies in $(0,w_+)$. The inverse branch of $K_\tau$ on $(0,w_+)$ maps $(b_\tau,\infty)$ real-analytically into $(0,w_+)$. Therefore $G_\tau$ has real boundary values on $(b_\tau,\infty)$, and so
\begin{equation}
    \operatorname{supp}\nu_\tau\cap(b_\tau,\infty)=\varnothing.
\end{equation}

For $z<0$, the integral representation of the $T$-transform gives $G_\tau(z)\in[-1,0)$. The physical branch is the inverse branch on $(-1,0)$, which maps $(-\infty,0)$ real-analytically into $(-1,0)$. As $z\uparrow0$, this branch tends to $w=-1$. Since $K_\tau(-1)=0$ and $K_\tau'(-1)=-e^{-\tau}\neq0$, the inverse branch continues real-analytically through $0$ into the interval $(0,a_\tau)$, with values in $(w_-,-1)$. Hence
\begin{equation}
    \operatorname{supp}\nu_\tau\cap(0,a_\tau)=\varnothing.
\end{equation}
The same limit excludes an atom at $0$: if $p=\nu_\tau(\{0\})$, then the integral representation gives $G_\tau(z)\to -(1-p)$ as $z\uparrow0$, while the inverse branch gives $G_\tau(z)\to -1$, so $p=0$. Together with the absence of mass on $(0,a_\tau)$, this shows that $0\notin\operatorname{supp}\nu_\tau$.

We have shown
\begin{equation}
    \operatorname{supp}\nu_\tau\subseteq[a_\tau,b_\tau].
\end{equation}
It remains to rule out gaps. Suppose that a nonempty open interval $I\subset(a_\tau,b_\tau)$ is disjoint from $\operatorname{supp}\nu_\tau$. Then $G_\tau$ extends real-analytically through $I$, and its boundary values on $I$ are real. The identity $K_\tau(G_\tau(z))=z$ continues across $I$, so
\begin{equation}
    K_\tau(G_\tau(x))=x,
    \qquad x\in I.
\end{equation}
This contradicts the displayed real-range identity above, because $G_\tau(x)$ is real while $x\in(a_\tau,b_\tau)$. Thus no open subinterval of $(a_\tau,b_\tau)$ is disjoint from the support. Since the support is closed, it follows that
\begin{equation}
    \operatorname{supp}\nu_\tau=[a_\tau,b_\tau]
    =[K_\tau(w_-),K_\tau(w_+)].
\end{equation}
\end{proof}

\section{Calculations for the Toy Random-Feature Model}
\label{sec:appendix-toy-ridge-risk}

This appendix records the elementary Gaussian and transform identities used in \cref{subsec:theory-practice-generalization}.

\begin{lemma}[Conditional Ridge Risk]
\label{lem:appendix-conditional-risk-general}
Fix a symmetric positive semidefinite matrix \(\Phi\in\mathbb R^{m\times m}\), and let \(\gamma>0\). Set \(A_\gamma:=\Phi(\Phi+\gamma I_m)^{-1}\). If \(f\mid\Phi\sim\mathcal N(0,C(\Phi))\), \(\varepsilon\sim\mathcal N(0,\sigma_\varepsilon^2I_m)\), and \(y=f+\varepsilon\), with \(f\) and \(\varepsilon\) conditionally independent given \(\Phi\), then
\begin{equation}
\label{eq:appendix-conditional-risk-general}
    \frac{1}{m}\mathbb E\!\left[\|A_\gamma y-f\|_2^2\mid\Phi\right]
    =
    \frac{1}{m}\Tr\!\left[
        \gamma^2C(\Phi)(\Phi+\gamma I_m)^{-2}
        +
        \sigma_\varepsilon^2\Phi^2(\Phi+\gamma I_m)^{-2}
    \right].
\end{equation}
\end{lemma}

\begin{proof}
Since \(A_\gamma-I_m=-\gamma(\Phi+\gamma I_m)^{-1}\), we have
\begin{equation}
\label{eq:appendix-risk-decomposition}
    A_\gamma y-f
    =
    (A_\gamma-I_m)f+A_\gamma\varepsilon.
\end{equation}
The cross term has conditional expectation zero by conditional independence and centering. Hence
\begin{equation}
\label{eq:appendix-risk-trace-proof}
\begin{aligned}
    \mathbb E\!\left[\|A_\gamma y-f\|_2^2\mid\Phi\right]
    &=
    \Tr\!\left[
        (A_\gamma-I_m)C(\Phi)(A_\gamma-I_m)
    \right]
    +
    \sigma_\varepsilon^2\Tr(A_\gamma^2) \\
    &=
    \Tr\!\left[
        \gamma^2C(\Phi)(\Phi+\gamma I_m)^{-2}
        +
        \sigma_\varepsilon^2\Phi^2(\Phi+\gamma I_m)^{-2}
    \right],
\end{aligned}
\end{equation}
where the last step uses cyclicity of the trace and that \(A_\gamma\) is a function of \(\Phi\).
\end{proof}

\begin{lemma}[Spectral Reduction]
\label{lem:appendix-feature-teacher-spectral-risk}
In the setting of \cref{lem:appendix-conditional-risk-general}, take \(C(\Phi)=\sigma_1^2\Phi\). If \(\rho_\Phi=m^{-1}\sum_i\delta_{\lambda_i(\Phi)}\), then
\begin{equation}
\label{eq:appendix-feature-teacher-spectral-risk}
    \frac{1}{m}\mathbb E\!\left[\|A_\gamma y-f\|_2^2\mid\Phi\right]
    =
    \int_{\mathbb R_{\geq0}}
    \frac{\sigma_1^2\gamma^2x+\sigma_\varepsilon^2x^2}{(x+\gamma)^2}\,
    \rho_\Phi(dx).
\end{equation}
Consequently, if the empirical spectral law converges weakly to \(\rho_\tau\) on the matched spectral-time scale and \(\gamma>0\) is fixed, then \eqref{eq:appendix-feature-teacher-spectral-risk} converges to \eqref{eq:limiting-spectral-risk}.
\end{lemma}

\begin{proof}
Diagonalize \(\Phi\). Since \(C(\Phi)=\sigma_1^2\Phi\), all matrices inside the trace in \eqref{eq:appendix-conditional-risk-general} are diagonal in the same eigenbasis. Therefore the normalized trace is the integral in \eqref{eq:appendix-feature-teacher-spectral-risk}. The integrand is bounded and continuous on \(\mathbb R_{\geq0}\) for fixed \(\gamma>0\), so weak convergence of the empirical spectral laws gives the limiting formula.
\end{proof}

\begin{lemma}[\(T\)-Transform Risk Identities]
\label{lem:appendix-risk-transform-identities}
Let \(G_\tau(z)=\int x(z-x)^{-1}\rho_\tau(dx)\). For \(\gamma>0\),
\begin{equation}
\label{eq:appendix-risk-transform-identities}
\begin{aligned}
    \int\frac{\gamma^2}{(x+\gamma)^2}\,\rho_\tau(dx)
    &=
    1+G_\tau(-\gamma)+\gamma\partial_zG_\tau(-\gamma), \\
    \int\frac{x}{(x+\gamma)^2}\,\rho_\tau(dx)
    &=
    -\partial_zG_\tau(-\gamma), \\
    \int\frac{x^2}{(x+\gamma)^2}\,\rho_\tau(dx)
    &=
    -G_\tau(-\gamma)+\gamma\partial_zG_\tau(-\gamma).
\end{aligned}
\end{equation}
Consequently, for the simplified feature-teacher model, \eqref{eq:limiting-spectral-risk} is equivalent to \eqref{eq:limiting-risk-transform}.
\end{lemma}

\begin{proof}
The derivative identity
\begin{equation}
\label{eq:appendix-G-derivative-risk}
    \partial_zG_\tau(z)
    =
    -\int\frac{x}{(z-x)^2}\,\rho_\tau(dx)
\end{equation}
gives the second line of \eqref{eq:appendix-risk-transform-identities} after setting \(z=-\gamma\). The first and third lines follow from
\begin{equation}
\label{eq:appendix-risk-algebra}
    \frac{\gamma^2}{(x+\gamma)^2}
    =
    1-\frac{x}{x+\gamma}-\frac{\gamma x}{(x+\gamma)^2},
    \qquad
    \frac{x^2}{(x+\gamma)^2}
    =
    \frac{x}{x+\gamma}-\frac{\gamma x}{(x+\gamma)^2}.
\end{equation}
Combining these identities with the simplified feature-teacher risk in \eqref{eq:limiting-spectral-risk} gives \eqref{eq:limiting-risk-transform}.
\end{proof}

\end{document}